\documentclass[reqno,a4paper]{amsart}

\usepackage[T1]{fontenc}
\usepackage{amssymb,amscd,amsthm,latexsym}
\allowdisplaybreaks[3]
\usepackage{graphicx}
\usepackage[british]{babel}
\usepackage[all]{xy}
\usepackage{xcolor}
\usepackage{mathbbol}
\usepackage{mathabx}
\usepackage{calrsfs}
\usepackage{booktabs}
\usepackage{mathtools}

\theoremstyle{plain}
\newtheorem{theorem}[subsection]{Theorem}
\newtheorem{lemma}[subsection]{Lemma}
\newtheorem{proposition}[subsection]{Proposition}
\newtheorem{corollary}[subsection]{Corollary}

\theoremstyle{definition}
\newtheorem{definition}[subsection]{Definition}
\newtheorem{convention}[subsection]{Convention}

\theoremstyle{remark}
\newtheorem{example}[subsection]{Example}

\newtheorem{remark}[subsection]{Remark}

\newenvironment{tfae}
{
\begin{enumerate}}
{\end{enumerate}}

\newdir{ >}{% end of arrow for the monos
 @{}*!/-10pt/@{>} }
\newdir{ |>}{% end of arrow for the kernels
 @{}*!/-5.5pt/@{|}*!/-10pt/:(1,-.2)@^{>}*!/-10pt/:(1,+.2)@_{>} }
\newdir{>>}{{}*!/3.5pt/:(1,-.2)@^{>}*!/3.5pt/:(1,+.2)@_{>}*!/7pt/:(1,-.2)@^{>}*!/7pt/:(1,+.2)@_{>}}
\newdir{ >>}{{}*!/8pt/@{|}*!/3.5pt/:(1,-.2)@^{>}*!/3.5pt/:(1,+.2)@_{>}}
\newdir{>}{{}*:(1,-.2)@^{>}*:(1,+.2)@_{>}}
\newdir{<}{{}*:(1,+.2)@^{<}*:(1,-.2)@_{<}}

\def\pullback{% with thanks to Valerian Even
 \ar@{-}[]+R+<6pt,-1pt>;[]+RD+<6pt,-6pt>%
 \ar@{-}[]+D+<1pt,-6pt>;[]+RD+<6pt,-6pt>}

\def\halfsplitpullback{%
 \ar@{-}[]+R+<6pt,-.5ex>;[]+RD+<6pt,-6pt>%
 \ar@{-}[]+D+<1pt,-6pt>;[]+RD+<6pt,-6pt>}

\def\ophalfsplitpullback{%
 \ar@{-}[]+R+<6pt,-1pt>;[]+RD+<6pt,-6pt>%
 \ar@{-}[]+D+<.5ex,-6pt>;[]+RD+<6pt,-6pt>}

\def\splitpullback{%
 \ar@{-}[]+R+<6pt,-.5ex>;[]+RD+<6pt,-6pt>%
 \ar@{-}[]+D+<1ex,-6pt>;[]+RD+<6pt,-6pt>}

\newcommand{\defn}[1]{\textbf{#1}}

\newcommand{\noproof}{\hfill\qed}
\newcommand{\To}{\Rightarrow}

\newcommand{\dashto}{\dashrightarrow}

\newcommand{\C}{\ensuremath{\mathbb{C}}}
\newcommand{\bi}[2]{\ensuremath{\lgroup #1 \;\, #2 \rgroup}}
\newcommand{\matriz}[4]{\ensuremath{\bigl\lgroup \begin{smallmatrix} #1 & #2 \\ #3 & #4 \end{smallmatrix}\bigr\rgroup}}
\newcommand{\Eng}{\ensuremath{\mathsf{Eng}}}
\newcommand{\Gp}{\ensuremath{\mathsf{Gp}}}
\newcommand{\Lie}{\ensuremath{\mathsf{Lie}}}
\newcommand{\Mon}{\ensuremath{\mathsf{Mon}}}
\newcommand{\HSLat}{\ensuremath{\mathsf{HSLat}}}
\newcommand{\Mag}{\ensuremath{\mathsf{Mag}}}
\newcommand{\Set}{\ensuremath{\mathsf{Set}}}
\newcommand{\PSet}{\ensuremath{\mathsf{Set}_*}}
\newcommand{\PSetop}{\ensuremath{\mathsf{Set}_*^{\op}}}
\newcommand{\Loop}{\ensuremath{\mathsf{Loop}}}
\newcommand{\Pt}{\ensuremath{\mathsf{Pt}}}
\newcommand{\SPt}{\ensuremath{\text{$\s$-$\Pt$}}}
\newcommand{\cod}{\ensuremath{\mathrm{cod}}}

\newcommand{\e}{\ensuremath{\mathfrak{e}}}
\newcommand{\x}{\ensuremath{\mathfrak{x}}}
\newcommand{\y}{\ensuremath{\mathfrak{y}}}

\newcommand{\K}{\ensuremath{\mathbb{K}}}

\newcommand{\Z}{\ensuremath{\mathbb{Z}}}
\newcommand{\V}{\ensuremath{\mathbb{V}}}
\newcommand{\X}{\ensuremath{\mathbb{X}}}
\newcommand{\s}{\ensuremath{\mathcal{S}}}
\newcommand{\ImC}{\ensuremath{\mathbb{K}}}
\newcommand{\op}{\ensuremath{\mathrm{op}}}
\newcommand{\Eq}{ \ensuremath{\mathrm{Eq}} }

\newcommand{\iS}[1]{{\bf (iS#1)}}
\newcommand{\iSs}[1]{{\bf (iSs#1)}}
\renewcommand{\S}[1]{{\bf (S#1)}}
\newcommand{\iL}[1]{{\bf (iL#1)}}

\newcommand{\ito}{\dasharrow}

\newcommand{\ci}{\circ}
\newcommand{\diam}{\circ}

\begin{document}

\title[Intrinsic Schreier special objects]{Intrinsic Schreier special objects}

\author[Andrea Montoli]{Andrea Montoli}
\address{Andrea Montoli, Dipartimento di Matematica ``Federigo Enriques'', Universit\`{a} degli
Studi di Milano, Via Saldini 50, 20133 Milano, Italy}{}
\email{andrea.montoli@unimi.it}

\author[Diana Rodelo]{Diana Rodelo}
\address{Diana Rodelo, Departamento de Matem\'atica, Faculdade de Ci\^{e}ncias e Tecnologia, Universidade do Algarve, Campus de Gambelas, 8005-139 Faro, Portugal and CMUC, Department of Mathematics, University of Coimbra, 3001-501 Coimbra, Portugal}
\thanks{The second author acknowledges financial support from the Centre for Mathematics of the University of Coimbra (UID/MAT/00324/2020, funded by the Portuguese Government through FCT/MCTES)}
\email{drodelo@ualg.pt}

\author[Tim Van~der Linden]{Tim Van~der Linden}
\address{Tim Van der Linden, Institut de Recherche en Math\'ematique et Physique, Universit\'e catholique de Louvain, che\-min du cyclotron~2 bte~L7.01.02, 1348 Louvain-la-Neuve, Belgium}
\thanks{The third author is a Research
Associate of the Fonds de la Recherche Scientifique--FNRS}
\email{tim.vanderlinden@uclouvain.be}



\begin{abstract}
Motivated by the categorical-algebraic analysis of split epimorphisms of monoids, we study the concept of a \emph{special object} induced by the \emph{intrinsic Schreier split epimorphisms} in the context of a regular unital category with binary co\-products, comonadic covers and a natural imaginary splitting in the sense of our article~\cite{ise}. In this context, each object comes naturally equipped with an imaginary magma structure. We analyse the intrinsic Schreier split epimorphisms in this setting, showing that their properties improve when the imaginary magma structures happen to be associative. We compare the intrinsic Schreier special objects with the protomodular objects, and characterise them in terms of the imaginary magma structure. We furthermore relate them to the Engel property in the case of groups and Lie algebras.
\end{abstract}

\date{\today}

\maketitle

\section{Introduction}
Recently, two different categorical approaches have been developed which aim to describe the homological properties of monoids, mainly in comparison with the properties groups have. The first one started with the observation that an important class of split epimorphisms of monoids, called \emph{Schreier split epimorphisms}, satisfies the convenient properties of split epimorphisms of groups \cite{SchreierBook, BM-FMS2}. The idea of considering Schreier split epimorphisms originated from the fact that these split epimorphisms correspond to monoid actions in the usual sense \cite{Patchkoria, MartinsMontoliSobral}. Although the category of monoids is not protomodular, Schreier split epimorphisms satisfy the properties that are typical for split epimorphisms in a protomodular category. This led to the notion of an \emph{$\s$-protomodular category}, with respect to a chosen class $\s$ of points---i.e., split epimorphisms with fixed section \cite{S-proto}. In an $\s$-protomodular category, it is always possible to identify a full subcategory which is protomodular~\cite{Bourn protomod}, called in~\cite{BM-FMS2} the \emph{protomodular core} with respect to the class $\s$. The objects of this subcategory are the \emph{$\s$-special objects}, namely those objects $X$ for which the split epimorphism $X \times X \leftrightarrows X,$ given by the second product projection and the diagonal morphism, belongs to~$\s$. The category of monoids is not protomodular but it is $\s$-protomodular with respect to the class of Schreier split epimorphisms, and its protomodular core is the category of groups.

The second approach consists in considering, in a pointed category with finite limits, a suitable class of objects, called \emph{protomodular objects} \cite{2Chs}. These are the objects $Y$ such that every split epimorphism with codomain $Y$ is \emph{stably strong}. A~split epimorphism with a given section is \emph{strongly split} if its kernel and its section are jointly extremal-epimorphic. It is \emph{stably strong} if every pullback of it along any morphism is a strongly split epimorphism. As proved in \cite{2Chs}, in the category of monoids the protomodular objects are precisely the groups.

The notion of protomodular object makes sense in every (pointed) category with finite limits, while Schreier special objects can apparently be considered only in the context of a J\'{o}nsson--Tarski variety \cite{JT}, because the notion of Schreier split epimorphism depends on the existence of a function, which is not a morphism in general, called the \emph{Schreier retraction}. In order to study this from a categorical perspective, we introduced in \cite{ise} the concept of \emph{intrinsic Schreier split epimorphism}, in the context of a regular unital category~\cite{B0} equipped with a comonadic cover (in the sense we recall in Subsection \ref{Comonadic covers}). This approach is inspired by the notion of imaginary morphism \cite{AMO}: indeed, the Schreier retraction we need is such an imaginary morphism. We showed in~\cite{ise} that these categories are $\s$-protomodular with respect to the class of intrinsic Schreier split epimorphisms, and we obtained an intrinsic version of the so-called Schreier special objects. It is shown in \cite{ise} that the concepts of intrinsic Schreier special object and protomodular object are independent. Since, however, the two coincide in the category of monoids, the question of understanding when the two notions are related arises naturally.

One of the goals of the present paper is to give an answer to this question. An important ingredient here is the observation that, when considering the Kleisli category associated with the comonad involved in the definition of an intrinsic Schreier split epimorphism, the definition itself simplifies greatly (Section~\ref{Intrinsic Schreier split epimorphisms}). Also, each object admits a canonical imaginary magma structure whose operation (called \emph{imaginary addition} in the text) depends on a choice of a natural imaginary splitting, which is part of our initial setting (Section~\ref{Imaginary addition in unital categories}). It turns out that an object is intrinsic Schreier special precisely when its imaginary magma structure is a one-sided loop structure (Theorem~\ref{one-sided loop}). Under the assumption that the imaginary addition is associative (Section~\ref{The associativity axiom}) we are able to extend several stability properties and homological lemmas which hold for Schreier extensions of monoids~\cite{SchreierBook} to our intrinsic context (Section~\ref{Stability properties}). Moreover, we prove that every intrinsic Schreier special object is a protomodular object (Corollary~\ref{iSchreier special => proto}).

It was shown in~\cite{ise} that there are only two possible choices for the natural imaginary splitting in the category of monoids, which leads to only two possible imaginary additions. This is no longer true for the category of groups or Lie algebras, where many options are available. Therefore, we focus on studying intrinsic Schreier special objects with respect to \emph{all} natural imaginary additions in these categories. We prove that $2$-Engel groups are intrinsic Schreier special with respect to all possible imaginary additions (Proposition~\ref{iSchreier for all t}). A similar result also holds for Lie algebras (Proposition~\ref{Thm:2EngelLie}).

\section{Imaginary morphisms}\label{Section imaginary morphs}
In this section we recall the concept of an \emph{imaginary morphism} which is of crucial importance in our work. We fix a particular setting where these imaginary morphisms can be defined.

\subsection{Imaginary morphisms~\cite{DB-ZJ-2009}}\label{Imaginary morphisms}
We take $\X$ to be the functor category $\Set^{\C^{\op}\times \C}$, where $\C$ is an arbitrary (small) category. Consider functors $\hom_{\C}$ and $A\colon\C^{\op}\times \C\to \Set$ and a natural transformation $\alpha\colon \hom_{\C} \To A$. If $\alpha$ is monomorphic, then all sets $A(X,Y)$ contain (an isomorphic copy of) $\hom_{\C}(X,Y)$. So, we may think of $A(X,Y)$ as an \emph{extension} of $\hom_{\C}(X,Y)$, and indeed in~\cite{DB-ZJ-2009} the triple $(\C, A, \alpha)$ was called an \emph{extended category}. The elements of $A(X,Y)\backslash \hom_{\C}(X,Y)$ will be called \defn{imaginary morphisms}. Sometimes it will be convenient to call a morphism in $\hom_{\C}(X,Y)$ a \defn{real morphism} to emphasise that it is an actual morphism in $\C$.

We use arrows of the type
$$
 \xymatrix{X \ar@{-->}[r] & Y}
$$
to represent an element of $A(X,Y)$, which could be an imaginary morphism or not. To distinguish those which are not, i.e., the elements of $A(X,Y)$ corresponding to a real morphism, say $f\colon X\to Y$, we tag the dashed arrow with the name of that real morphism overlined (instead of $\alpha_{X,Y}(f)$):
$$
 \xymatrix{X \ar@{-->}[r]^-{\overline{f}} & Y.}
$$

It is possible to define an extended composition, denoted by $\ci$, be\-tween real and imaginary morphisms as follows:

$$
 \xymatrix{X \ar@{-->}[r]^-{a} \ar@{-->}@/_1pc/[rr]_-{v\ci a} & Y \ar[r]^-v & V,} \;\;\text{where}\;\; v\ci a=A(1_X,v)(a)
$$
and
$$
 \xymatrix{U \ar[r]^-u \ar@{-->}@/_1pc/[rr]_-{a\ci u} & X \ar@{-->}[r]^-{a} & Y,} \;\;\text{where}\;\; a\ci u=A(u,1_Y)(a).
$$

If $a$ corresponds to a real morphism, i.e., $a=\overline{f}=\alpha_{X,Y}(f\colon X\to Y)$, then the same is true for $v\ci a$ and $a\ci u$. Indeed, by the naturality of $\alpha$ we have
$$
 A(1_X,v)(\alpha_{X,Y}(f)) = \alpha_{X,V}(vf),
$$
so that $v\ci \overline{f} = \overline{vf}(=\alpha_{X,V}(vf))$ corresponds to the real morphism $vf$. Similarly, $\overline{f}\ci u = \overline{fu}(=\alpha_{U,Y}(fu))$ corresponds to the real morphism $fu$. In particular, we obtain identity properties $v\ci \overline{1_Y}=\overline{v}$ and $\overline{1_X}\ci u=\overline{u}$. There is also an associativity property, which follows from the fact that $A$ is a functor:
\begin{align*}
(v\ci a)\ci u & = A(u,1_V)(A(1_X,v)(a))=A(u,v)(a)\\
 & =A(1_U,v)(A(u,1_Y)(a))=v\ci (a\ci u).	
\end{align*}
 
\begin{definition}\label{def imaginary splitting}
We say that a real morphism $f\colon X\to Y$ admits an \defn{imaginary splitting} when there exists an imaginary morphism $s$ such that the following diagram commutes
$$
\xymatrix{Y \ar@{-->}[r]^-{s} \ar@{-->}@/_1pc/[rr]_-{f\ci s=\overline{1_Y}} & X \ar[r]^-f & Y.}
$$
\end{definition}

\subsection{Comonadic covers}\label{Comonadic covers}
We assume that $\C$ is a regular category equipped with a comonad $(P,\delta,\varepsilon)$ whose counit $\varepsilon$ is a regular epimorphism. Then for each object $X$ in $\C$, the morphism $\varepsilon_X\colon {P(X)\twoheadrightarrow X}$ is a \emph{(comonadically) chosen cover} of $X$ in~$\C$, which for us means that we have a regular epimorphism $\varepsilon_X$ with codomain~$X$, determined by the given comonad. 

Note that for any morphism $f\colon X\to Y$ in $\C$
\begin{equation}\label{naturality of epsilon}
 f\varepsilon_X=\varepsilon_YP(f)
\end{equation}
and
\begin{equation}\label{naturality of delta}
 P^2(f)\delta_X=\delta_YP(f),
\end{equation}
where $P^2=PP$. Also
\begin{equation}\label{counit}
 \varepsilon_{P(X)} \delta_X=1_{P(X)}=P(\varepsilon_X)\delta_X
\end{equation}
and
\begin{equation}\label{coassociative}
 P(\delta_X)\delta_X = \delta_{P(X)}\delta_X,
\end{equation}
for any object $X$ in $\C$.

\begin{example}\label{Varieties Comonad}
If $\V$ is a variety of universal algebras, then we may consider the free algebra comonad $(P,\delta,\varepsilon)$. For any algebra $X$, we have
\[
\varepsilon_X \colon P(X) \twoheadrightarrow X \colon [x] \mapsto x \qquad \text{and} \qquad \delta_X \colon P(X) \hookrightarrow P^2(X)\colon[x] \mapsto \left[ [x] \right],
\]
where $[x]$ denotes the one letter word $x$; such words are the generators of $P(X)$. In this case, any function $f\colon{X\to Y}$ between algebras $X$ and $Y$ extends uniquely to a morphism $P(X) \to Y\colon \left[x\right] \mapsto f(x)$ in $\V$.
\end{example}

\subsection{Imaginary morphisms induced from comonadic covers}\label{Imaginary morphisms induced from comonadic covers}
The idea behind functions extending to real morphisms in Example~\ref{Varieties Comonad} can be captured through the notion of imaginary morphism: it is like a function (not a morphism) $X \dashto Y$ of algebras $X$ and $Y$ that extends to an actual morphism of algebras $P(X)\to Y$. More precisely, given a regular category $\C$ with comonadic covers we define the functor
$$
\begin{array}{rcl} A: \C^{\op}\times \C & \rightarrow & \Set, \\
 (X,Y) & \mapsto & \hom_{\C}(P(X),Y) \vspace{3pt}\\
 u \uparrow\; \downarrow v & & \downarrow A(u,v)\vspace{3pt}\\
 (U,V) & \mapsto & \hom_{\C}(P(U),V)
\end{array}
$$
where $A(u,v)=\hom_{\C}(P(u),v)$. So, $A$ is just the functor $\hom_{\C} (P^{\op}\times 1_{\C})$.

The components of $\alpha\colon \hom_{\C}\Rightarrow A$ are defined, for all objects $X$, $Y$ by
\[
\alpha_{X,Y}\colon \hom_{\C}(X,Y) \rightarrow \hom_{\C}(P(X),Y)\colon \bigl( X \stackrel{f}{\rightarrow} Y\bigr) \mapsto \bigl(P(X) \stackrel{\varepsilon_X}{\twoheadrightarrow} X \stackrel{f}{\rightarrow} Y\bigr).
\]
Note that $\alpha$ is indeed a natural transformation because $\varepsilon$ is (see \eqref{naturality of epsilon}). Also, the components $\alpha_{X,Y}$ are injective, for all objects $X,Y$, since $\varepsilon_X$ is a regular epimorphism. Since the elements of $A(X,Y)=\hom_{\C}(P(X),$ $Y)$ are actual morphisms in $\C$, an arrow of the type $X \dashto Y$ corresponds to a morphism $P(X)\to Y$. According to Subsection~\ref{Imaginary morphisms}:

\begin{itemize}
\item if $X\stackrel{f}{\rightarrow} Y$ is a real morphism, then $X\stackrel{\overline{f}}{\dashto} Y$ corresponds to the morphism $P(X) \stackrel{f\varepsilon_X}{\longrightarrow} Y$, so $\overline{f}=f\varepsilon_X$;
\item an imaginary morphism $X\stackrel{a}{\dashto} Y$ is a (real) morphism $P(X)\stackrel{a}{\rightarrow} Y$ which is \emph{not} of the form $a=f\varepsilon_X$, for some real morphism $X\stackrel{f}{\rightarrow} Y$;
\item the composition of a real morphism with an imaginary one is defined by
\[
\begin{cases}
\xymatrix{X \ar@{-->}[r]^-{a} \ar@{-->}@/_1pc/[rr]_-{v\ci a} & Y \ar[r]^-v & V}, & \text{where $v\ci a=va\colon P(X) \rightarrow V$,}\\
\xymatrix{U \ar[r]^-u \ar@{-->}@/_1pc/[rr]_-{a\ci u} & X \ar@{-->}[r]^-{a} & Y}, & \text{where $a\ci u=aP(u)\colon P(U)\rightarrow Y$.}
\end{cases}
\]
\end{itemize}

\begin{convention}
From now on, we only consider imaginary morphisms that are induced from comonadic covers.
\end{convention}

\begin{remark}\label{rem imaginary splitting}
It is clear that in this setting the existence of an imaginary splitting (Definition~\ref{def imaginary splitting}) for a morphism $f$ implies that $f$ is a regular epimorphism ($f\ci s=\overline{1_Y}$ implies that $fs=\varepsilon_Y$, which is a regular epimorphism). The converse holds when the values of $P$ are projective objects in $\C$. If $f\colon X\twoheadrightarrow Y$ is a regular epimorphism, then $f$ admits an imaginary splitting because $P(Y)$ is projective
$$
\xymatrix{& P(Y) \ar[dl]_-{\exists\; s}\ar@{>>}[d]^-{\varepsilon_Y} \\
 X \ar@{>>}[r]_-f & Y;}
$$
thus $fs=\varepsilon_{Y}$. So the existence of imaginary splittings characterises regular epimorphisms in this setting. Moreover, $P(f)$ is a split epimorphism since
\[
P(f)P(s)\delta_Y=P(\varepsilon_Y)\delta_Y\stackrel{\eqref{counit}}{=} 1_{P(Y)}.
\]
\end{remark}

\section{The Kleisli category}\label{The Kleisli category} Let $\C$ be a regular category with comonadic covers. We denote by $\ImC$ the Kleisli category associated to the comonad $(P, \delta, \varepsilon)$. Its objects are those of $\C$ and $\hom_{\ImC}(X,Y)=\hom_{\C}(P(X),Y)$. The morphisms of $\ImC$ are the imaginary morphisms together with those of the type $\overline{f}\colon X\dashto Y$, for some real morphism $f\colon X\to Y$ (Subsection~\ref{Imaginary morphisms induced from comonadic covers}).

The composition in $\ImC$ will also be denoted by $\diam$ (as in Subsections~\ref{Imaginary morphisms induced from comonadic covers} and~\ref{Imaginary morphisms})
$$
\xymatrix{A \ar@{-->}[r]^a \ar@{-->}@/_1pc/[rr]_-{b\diam a} & B \ar@{-->}[r]^-b & C,}
$$
where $b\diam a$ corresponds to the morphism in $\C$
$$
\xymatrix{P(A) \ar[r]^-{\delta_A} & P^2(A) \ar[r]^-{P(a)} & P(B) \ar[r]^-b & C.}
$$

\begin{remark}\label{pps of diamond}
If any of the morphisms in a composite in $\ImC$ corresponds to a real morphism, then this composite coincides with the one defined in Subsections~\ref{Imaginary morphisms induced from comonadic covers} and~\ref{Imaginary morphisms}---See Table~\ref{Fig:Compositions}.
\begin{table}
\begin{tabular}{c@{\qquad}c}
\toprule
{composition $\ci$} & {composition in $\C$} \\
\hline
$\xymatrix{X \ar@{-->}[r]^-{a} \ar@{-->}@/_1pc/[rr]_-{v\circ a} & Y \ar[r]^-v & V}$ & $\xymatrix{P(X) \ar[r]^-{a} & Y \ar[r]^-v & V}$\\
$\xymatrix{X \ar@{-->}[r]^-{a} \ar@{-->}@/_1pc/[rr]_-{\overline{v}\circ a} & Y \ar@{-->}[r]^-{\overline{v}} & V}$ & $\xymatrix{P(X) \ar[r]^-{\delta_X} \ar@{=}[drr]^-{\eqref{counit}} & P^2(X) \ar[r]^-{P(a)} \ar[dr]^-{\varepsilon_{P(X)}} & P(Y) \ar[r]^-{\varepsilon_Y} \ar@{}[d]|-{\eqref{naturality of epsilon}} & \ar[r]^-v & V \\ & & P(X) \ar[ur]^-a}$\\
$\xymatrix{U \ar[r]^-{u} \ar@{-->}@/_1pc/[rr]_-{a\circ u} & X \ar@{-->}[r]^-a & Y}$ & $\xymatrix{P(U) \ar[r]^-{P(u)} & P(X) \ar[r]^-a & Y}$\\
$\xymatrix{U \ar@{-->}[r]^-{\overline{u}} \ar@{-->}@/_1pc/[rr]_-{a\circ \overline{u}} & X \ar@{-->}[r]^-a & Y}$ & $\xymatrix{P(U) \ar[r]^-{\delta_U} \ar@{=}@/_2pc/[rr]^-{\eqref{counit}} & P^2(U) \ar[r]^-{P(\varepsilon_U)} & P(U) \ar[r]^-{P(u)} & P(X) \ar[r]^-a & Y}$\\
\bottomrule
\end{tabular}
\medskip
\caption{Composition in the Kleisli category}\label{Fig:Compositions}
\end{table}
\end{remark}

The comonad $(P,\delta, \varepsilon)$ gives rise to an adjunction $\ImC \rightleftarrows \C$, where the right adjoint is the embedding
\[
 I \colon \C \rightarrow \ImC \colon
 \xymatrix{X \ar[r]^-f & Y} \mapsto \xymatrix{X \ar@{-->}[r]^-{\overline{f}} & Y}
\]
Consequently, $\ImC$ has a limit for every finite diagram in $\C$, which is just the limit of that diagram in $\C$, embedded into $\ImC$.

\section{Imaginary addition in unital categories}\label{Imaginary addition in unital categories}
In this section we define an \emph{imaginary addition} on each object $X$ of a unital category with comonadic covers, i.e., an imaginary morphism $\mu^X\colon X\times X\dashto X$ such that $\mu^X\ci \langle 1_X,0\rangle =\overline{1_X}$ and $\mu^X\ci \langle 0,1_X\rangle =\overline{1_X}$. Such an imaginary addition provides one of the tools needed to define intrinsic Schreier split epimorphisms in Section~\ref{Intrinsic Schreier split epimorphisms}.

\subsection{Unital categories~\cite{B0}}\label{Unital categories}
A pointed and finitely complete category is called \defn{unital} when, for all objects $A, B$,
$$
\xymatrix@!0@=5em{ A \ar[r]^-{\langle 1_A,0\rangle} & A\times B & B \ar[l]_-{\langle 0,1_B\rangle}}
$$
is a jointly extremal-epimorphic pair.

\begin{example}
As shown in~\cite{Borceux-Bourn}, a variety of universal algebras $\V$ is unital if and only if it is a \defn{J\'onsson--Tarski variety}~\cite{JT}. Recall that a J\'onsson--Tarski variety is such that its theory contains a unique constant $0$ and a binary operation $+$ satisfying the identities $x+0=x=0+x$. So an algebra is a unitary magma, possibly equipped with additional operations.
\end{example}

A pointed finitely complete category $\C$ is unital if and only if for any \defn{punctual span}
\[
\xymatrix@!0@=5em{ A \ar@<.5ex>[r]^-{s} & C \ar@<-.5ex>@{>>}[r]_-{g} \ar@<+.5ex>@{>>}[l]^-{f} & B, \ar@<-.5ex>[l]_-{t}} \qquad\qquad
 \text{$fs=1_A$, $gt=1_B$, $ft=0$, $gs=0$}
\]
in $\C$, the induced morphism $\langle f,g\rangle\colon C\twoheadrightarrow A\times B$ is a strong epimorphism (Theorem 1.2.12 in~\cite{Borceux-Bourn}). Consequently, a pointed regular category with binary coproducts is unital if and only if for all objects $A,B$, the comparison morphism
$$
r_{A,B}=\matriz{1_A}{0}{0}{1_B}\colon A+B\twoheadrightarrow A\times B
$$
is a regular epimorphism.

\subsection{Natural imaginary splittings~\cite{ise}}\label{Natural imaginary splittings}
If $\C$ is a regular unital category with binary coproducts and comonadic covers, then for all objects $A,B$, the comparison morphism $r_{A,B}=\matriz{1_A}{0}{0}{1_B}\colon A+B \twoheadrightarrow A\times B$ is a regular epimorphism. When $P(A\times B)$ is a projective object, as in the varietal case, there exists a (not necessarily unique) morphism $t_{A,B}\colon P(A\times B)\to A+B$ such that
\begin{equation}\label{im splitting}
 r_{A,B} t_{A,B}=\varepsilon_{A\times B}
\end{equation}
(see Remark~\ref{rem imaginary splitting}). That is to say, there exists an imaginary splitting $t_{A,B}$ for the regular epimorphism $r_{A,B}$. 

\begin{example}\label{LeftRightImaginarySplitting}
Let $\V$ be a J\'{o}nsson--Tarski variety. For any pair of algebras $(A,B)$ in~$\V$, we can make the following choices of an imaginary splitting for $r_{A,B}$: the \emph{direct imaginary splitting} $t^d_{A,B}$
\[
[(a,b)]\mapsto \underline{a}+\overline{b}
\]
which sends a generator $[(a,b)]\in P(A\times B)$ to the sum of $\underline{a}=\iota_1(a)$ with $\overline{b}=\iota_2(b)$ in~${A+B}$ (where $\iota_1$ and $\iota_2$ are the coproduct inclusions); and the \emph{twisted imaginary splitting} $t^w_{A,B}$
\[
[(a,b)]\mapsto \overline{b}+\underline{a}
\]
which does the same, but in the opposite order. Note that both of those choices $t_{A,B}$ are natural in $A$ and in $B$, so that they each determine a natural transformation
\[
t\coloneq (t_{A,B}\colon P(A\times B)\to A+B)_{A,B\in \C}
\]
such that $r t=\varepsilon$ as natural transformations.

It was shown in~\cite{ise} that when $\V$ is the category $\Mon$ of monoids, then the above choices of natural imaginary splittings (direct and twisted) are the only options. This is far from being true in general: if $A$ and $B$ are groups then we can also send $[(a,b)]$ to $\underline{-a} + \overline{b} + \underline{a+a}$, for instance. See the end of Section~\ref{isso vs proto objs} for further examples.
\end{example}

We make the existence of a natural $t$ into an axiom. Let $\C$ be a pointed regular (unital) category $\C$ with binary coproducts and comonadic covers. Suppose also that there exist $t_{A,B}$ such that \eqref{im splitting} holds and that they are the components of a natural transformation $t$, where $r t=\varepsilon$. Then all $r_{A,B}$ are necessarily regular epimorphisms (because the $\varepsilon_{A\times B}$ are) and, consequently, $\C$ is a unital category. In~\cite{ise} such a natural transformation $t$ was called a \defn{natural imaginary splitting}.

\begin{remark}\label{remarks on t}
Any natural imaginary splitting $t= (t_{A,B}\colon A\times B\ito A+B)_{A,B\in \C}$ has the following properties:
\begin{itemize}
\item[1.] $t_{A,0}$ can be identified with $\varepsilon_A$, up to canonical isomorphisms, as follows:
$$
\xymatrix@=30pt{P(A) \ar@{>>}[r]^-{\varepsilon_A} \ar[d]|-{\cong} & A \ar@{=}[r]^-{1_A} \ar[d]|-{\cong} & A \ar[d]|-{\cong} \\
 P(A\times 0) \ar[r]^-{t_{A,0}} \ar@{>>}@/_1pc/[rr]_-{\varepsilon_{A\times 0}} & A+0 \ar@{>>}[r]^-{r_{A,0}} & A\times 0,}
$$
for all objects $A$ in $\C$;
\item[2.] the naturality of $t$ gives the commutative diagram
\begin{equation}\label{naturality of t}
\vcenter{\xymatrix@=30pt{P(A\times B) \ar[r]^-{t_{A,B}} \ar[d]_-{P(u\times v)} & A+B \ar[d]^-{u+v} \\
 P(C\times D) \ar[r]_-{t_{C,D}} & C+D}}
\end{equation}
for all $u\colon A\to C$, $v\colon B\to D$ in $\C$;
\item[3.] from~\eqref{im splitting}, we deduce
\begin{equation}\label{pp1 for t}
 \bi{1_A}{0}t_{A,B}=\pi_A\varepsilon_{A\times B} \stackrel{\eqref{naturality of epsilon}}{=}\varepsilon_A P(\pi_A)
\end{equation}
and
\begin{equation}\label{pp2 for t}
 \bi{0}{1_B}t_{A,B}=\pi_B\varepsilon_{A\times B} \stackrel{\eqref{naturality of epsilon}}{=}\varepsilon_B P(\pi_B)
\end{equation}
for all objects $A$ and $B$ in $\C$;
\item[4.] using properties 1.\ and 2.\ above, we obtain the (regular epimorphism, monomorphism) factorisations
\begin{equation}\label{pp3 for t}
\vcenter{\xymatrix@C=50pt@R=10pt{P(A) \ar[r]^-{P(\langle 1_A,0\rangle)} \ar@{>>}[dr]_-{\varepsilon_A} & P(A\times B) \ar[r]^-{t_{A,B}} & A+B\\
 & A \ar@{ >->}[ru]_-{\iota_1}}}
\end{equation}
and
\begin{equation}\label{pp4 for t}
\vcenter{\xymatrix@C=50pt@R=10pt{P(B) \ar[r]^-{P(\langle 0, 1_B\rangle)} \ar@{>>}[dr]_-{\varepsilon_B} & P(A\times B) \ar[r]^-{t_{A,B}} & A+B, \\
 & B \ar@{ >->}[ru]_-{\iota_2}}}
\end{equation}
for all objects $A$ and $B$ in $\C$.
\end{itemize}
\end{remark}

\subsection{Imaginary addition}\label{Imaginary addition}
Let $\C$ be a regular unital category with binary coproducts, comonadic covers and a natural imaginary splitting $t$. For every object $X$, we consider the imaginary morphism $\mu^X\colon X\times X \dashto X$ given by
\begin{equation}\label{imaginary addition}
 \xymatrix@C=45pt{P(X\times X) \ar[r]^-{t_{X,X}} \ar@{-->}@/_1pc/[rr]_-{\mu^X} & X+X \ar[r]^-{\bi{1_X}{1_X}} & X.}
\end{equation}
We call $\mu^X$ an \defn{imaginary addition} on $X$ since
\begin{equation}\label{x+0=x}
 \xymatrix@C=45pt{X \ar[r]^-{\langle 1_X,0\rangle} \ar@{-->}@/_1pc/[rr]_-{\mu^X\ci \langle 1_X,0\rangle=\overline{1_X}} & X\times X \ar@{-->}[r]^-{\mu^X} & X}
\end{equation}
and
\begin{equation}\label{0+x=x}
 \xymatrix@C=45pt{X \ar[r]^-{\langle 0, 1_X\rangle} \ar@{-->}@/_1pc/[rr]_-{\mu^X\ci \langle 0,1_X\rangle=\overline{1_X}} & X\times X \ar@{-->}[r]^-{\mu^X} & X.}
\end{equation}
Indeed,
$
 \mu^X\ci \langle 1_X,0\rangle=\bi{1_X}{1_X} t_{X,X} P(\langle 1_X,0 \rangle) \stackrel{\eqref{pp3 for t}}{=} \bi{1_X}{1_X}\iota_1\varepsilon_X = \varepsilon_X = \overline{1_X}
$
and
$
 \mu^X\ci \langle 0,1_X\rangle=\bi{1_X}{1_X} t_{X,X} P(\langle 0,1_X \rangle) \stackrel{\eqref{pp4 for t}}{=} \bi{1_X}{1_X}\iota_2\varepsilon_X = \varepsilon_X = \overline{1_X}.
$
We adapt Definition 3.15 in~\cite{DB-ZJ-2009} to the unital context and call the family
\[
\mu\coloneq(\mu^X\colon X\times X \dashto X)_{X\in \C}
\]
a \defn{natural addition}. Here ``natural'' means that for any morphism $f\colon X\to Y$ the diagram
\begin{equation}\label{naturality of mu}
\vcenter{\xymatrix@=30pt{ X\times X \ar@{-->}[r]^-{\mu^X} \ar[d]_-{f\times f} & X \ar[d]^-f \\
 Y\times Y \ar@{-->}[r]_-{\mu^Y} & Y}}
\end{equation}
commutes. Indeed,
\begin{align*}
f\ci \mu^X &\stackrel{\phantom{\eqref{naturality of t}}}{=} f \bi{1_X}{1_X} t_{X,X} = \bi{1_Y}{1_Y} (f+f)t_{X,X} \\&\stackrel{\eqref{naturality of t}}{=} \bi{1_Y}{1_Y} t_{Y,Y} P(f\times f) = \mu^Y\ci (f\times f).
\end{align*}

\section{Intrinsic Schreier split extensions}\label{Intrinsic Schreier split epimorphisms}
In this section we recall the notion of a Schreier split epimorphism of monoids and its extended categorical version, the notion of an intrinsic Schreier split epimorphism. We actually give a simplified version of the intrinsic definition by using the direct composition of imaginary morphisms, which is simply the composition in the Kleisli category associated with the comonad of the comonadic covers.

\subsection{Schreier split extensions of monoids~\cite{SchreierBook, BM-FMS2}}\label{Schreier split extensions of monoids}
We recall the definition and the main properties concerning Schreier split epimorphisms.

A split epimorphism of monoids $f$ with chosen section $s$ and kernel $k$
\begin{equation}\label{split ext of monoids}
 \xymatrix@!0@=5em{ K \ar@{ |>->}[r]_-{k} & (X,\cdot,1) \ar@<-.5ex>@{>>}[r]_-{f} & Y \ar@<-.5ex>[l]_-{s}}
\end{equation}
is called a \defn{Schreier split epimorphism} if, for every $x\in X$, there exists a unique element $a\in K$ such that $x=k(a)\cdot sf(x)$. Equivalently, if there exists a unique function $q\colon X\dashto K$ such that $x=kq(x)\cdot sf(x)$ for all $x\in X$. We emphasise the fact that $q$ is just a function (not necessarily a morphism of monoids) by using an arrow of type $\dashto$.

The uniqueness property may be replaced~\cite[Proposition~2.4]{BM-FMS2} by an extra condition on $q$: the couple $(f,s)$ as in \eqref{split ext of monoids} is a Schreier split epimorphism if and only if there exists a $q$, as above, such that
\begin{itemize}
\item[\S1] $x=kq(x)\cdot sf(x)$, for all $x\in X$;
\item[\S2] $q(k(a)\cdot s(y))=a$, for all $a\in K$, $y\in Y$.
\end{itemize}

\begin{remark}\label{homogeneous}Recall from~\cite{SchreierBook}
that Schreier split epimorphisms are also called \defn{right homogenous} split epimorphisms. A split epimorphism as in~\eqref{split ext of monoids} is called \defn{left homogenous} if, for every $x\in X$, there exists a unique element $a\in K$ such that $x=sf(x)\cdot k(a)$.
\end{remark}

\begin{proposition}\cite[Proposition~2.1.5]{SchreierBook} \label{basic pps}
Given a Schreier split epimorphism as in~\eqref{split ext of monoids}, the following hold:
\begin{itemize}
\item[\S3] $qk=1_K$;
\item[\S4] $qs=0$;
\item[\S5] $q(1)=1$;
\item[\S6] $kq(s(y)\cdot k(a))\cdot s(y)=s(y)\cdot k(a)$, for all $a\in K$, $y\in Y$;
\item[\S7] $q(x\cdot x')=q(x)\cdot q(sf(x)\cdot kq(x'))$, for all $x$, $x'\in X$.
\end{itemize}
\end{proposition}

A split epimorphism as in~\eqref{split ext of monoids} is said to be \defn{strong} when $(k,s)$ is a jointly extremal-epimorphic pair. It is \defn{stably strong} if every pullback of it along any morphism is strong. Any Schreier split epimorphism is (stably) strong (see~\cite{SchreierBook}, Lemma 2.1.6 and Proposition 2.3.4), thus $f$ is the cokernel of its kernel $k$. So, viewed as a split short exact sequence, such a split epimorphism in fact forms a \emph{Schreier split extension}.

As shown in \cite{MartinsMontoli}, the definition of a Schreier split epimorphism makes sense also in the wider context of J\'onsson--Tarski varieties.

\subsection{Intrinsic Schreier split extensions~\cite{ise}}\label{ISSE}
We recall our approach to Schreier extensions. Here $\C$~will denote a regular unital category with binary coproducts, comonadic covers and a natural imaginary splitting $t$.

\begin{definition} A split epimorphism $f$ with chosen section $s$ and kernel $k$
\begin{equation}\label{iSchreier}
\xymatrix@!0@=4em{ K \ar@{ |>->}[r]_-{k} & X \ar@<-.5ex>@{>>}[r]_-{f} & Y, \ar@<-.5ex>[l]_-{s}}
\end{equation}
is called an \defn{intrinsic Schreier split epimorphism} (with respect to $t$) if there exists an imaginary morphism $q\colon {X\ito K}$ (i.e., a morphism $q \colon P(X) \to K$), called the \defn{imaginary (Schreier) retraction}, such that
\begin{itemize}
\item[\iS1] $\mu^X\diam \langle k\ci q,\overline{sf}\rangle=\overline{1_X}$, i.e., the diagram
$$
\xymatrix@=30pt{X \ar@{-->}[r]^-{\langle k\ci q, \overline{sf}\rangle} \ar@{-->}[dr]_-{\overline{1_X}} & X\times X \ar@{-->}[d]^-{\mu^X} \\ & X}
$$
commutes;
\item[\iS2] $q\diam \mu^X \ci (k\times s)=\overline{\pi_K}$, i.e., the diagram
$$
\xymatrix@=30pt{ K\times Y \ar[r]^-{k\times s} \ar@{-->}[drr]_-{\overline{\pi_K}} & X\times X \ar@{-->}[r]^-{\mu^X} & X \ar@{-->}[d]^-q \\ & & K}
$$
 commutes.
\end{itemize}
\end{definition}

The original definition in~\cite{ise} expressed the above axioms through their corresponding morphisms and equalities in $\C$. However, using the composition in the Kleisli category $\ImC$, as above, gives a better understanding of the link with \S1 and~\S2.

The imaginary retraction of an intrinsic Schreier split epimorphism is necessarily unique (see~\cite[Proposition 5.3]{ise}) and we also have (by~\cite[Proposition 5.4]{ise}):
\begin{itemize}
\item[\iS3] $\xymatrix@=30pt{K \ar[r]^-k \ar@{-->}@/_1pc/[rr]_-{q\ci k=\overline{1_K}} & X \ar@{-->}[r]^-{q} & K,}$ i.e., $qP(k)=\varepsilon_K$;
\item[\iS4] $\xymatrix@=30pt{Y \ar[r]^-s \ar@{-->}@/_1pc/[rr]_-{q\ci s=\overline{0}} & X \ar@{-->}[r]^-{q} & K,}$ i.e., $qP(s)=0$;
\item[\iS5] $\xymatrix@=30pt{0 \ar[r]^-{0_X} \ar@{-->}@/_1pc/[rr]_-{q\ci 0_X=\overline{0_K}} & X \ar@{-->}[r]^-{q} & K,}$ i.e., $qP(0_X) (=q0_{P(X)})=0_K$;
\item[\iS6] $\xymatrix@C=100pt@R=30pt{Y\times K \ar[d]_-{s\times k} \ar@{-->}[r]^-{\langle k\circ q\circ \mu^X\circ (s\times k), \overline{s\pi_Y}\rangle}
 & X\times X \ar@{-->}[d]^-{\mu^X}\\
 X\times X \ar@{-->}[r]_-{\mu^X} & X, }$\\ i.e., $\mu^X\circ \langle k\circ q\circ \mu^X\circ (s\times k), \overline{s\pi_Y}\rangle = \mu^X \circ (s\times k)$.
\end{itemize}

In order to obtain an intrinsic version of \S7, we will need a further assumption, which will be discussed in the next section.

If we apply this intrinsic definition to the category $\Mon$ of monoids, we regain the original definition of a Schreier split epimorphism (=~right homogeneous split epimorphism). Also, left homogeneous split epimorphisms (see Remark~\ref{homogeneous}) fit the picture. Indeed:

\begin{theorem}~\cite[Theorem 5.10]{ise}
In the case of monoids, the intrinsic Schreier split epimorphisms with respect to the direct imaginary splitting $t^d$ are precisely the Schreier split epimorphisms. Similarly, the intrinsic Schreier split epimorphisms with respect to the twisted imaginary splitting $t^w$ are the left homogeneous split epimorphisms.\noproof
\end{theorem}

This result extends to J\'onsson--Tarski varieties~\cite{ise}.

\subsection{$\s$-protomodular categories~\cite{SchreierBook}}

We now recall the definition of an $\s$-protomodular category, with respect to a class $\s$ of points (i.e., of split epimorphisms with a fixed section) in a pointed category $\C$ with finite limits.

We denote by $\Pt(\C)$ the category of points in $\C$, whose morphisms are pairs of morphisms which form commutative squares with both the split epimorphisms and their sections. The functor $\cod\colon \Pt(\C) \to \C$ associates with every split epimorphism its codomain. It is a fibration, usually called the \defn{fibration of points}. For each object $Y$ of $\C$, we denote by $\Pt_Y(\C)$ the fibre of this fibration, whose objects are the points with codomain $Y$.

Let $\s$ be a class of points in $\C$ which is stable under pullbacks along any morphism. If we look at it as a full
subcategory $\SPt(\C)$ of $\Pt(\C)$, then it gives rise to a subfibration $\s$-$\cod$ of the fibration of points.

\begin{definition}\cite[Definition~8.1.1]{SchreierBook}\label{S-protomodular category}
Let $\C$ be a pointed finitely complete category, and $\s$ a
pullback-stable class of points. We say that $\C$ is
\defn{$\s$-proto\-modular} when:
\begin{enumerate}
\item every point in $\SPt(\C)$ is a strong point;
\item $\SPt(\C)$ is closed under finite limits in $\Pt(\C)$.
\end{enumerate}
\end{definition}

As shown in~\cite{S-proto}, $\s$-protomodular categories satisfy, relatively to the class $\s$, many of the properties of protomodular categories~\cite{Bourn protomod}. In particular, a relative version of the Split Short Five Lemma holds: given a morphism of $\s$-split extensions, i.e., a diagram
$$
\xymatrix@=40pt{K' \ar@{ |>->}[r]^-{k'} \ar[d]_-{\gamma} & X' \ar@{>>}@<-.5ex>[r]_-{f'} \ar[d]_-{g}
 & Y' \ar@<-.5ex>[l]_-{s'} \ar[d]^-{h}\\
 K \ar@{ |>->}[r]^-{k} & X \ar@{>>}@<-.5ex>[r]_-{f} & Y, \ar@<-.5ex>[l]_-{s}}
$$
such that the two rows are $\s$-split extensions (points in $\s$ with their kernel) and the three squares involving, respectively, the split epimorphisms, the kernels, and the sections commute, $g$ is an isomorphism if and only if both $\gamma$ and $h$ are isomorphisms. In Section \ref{Stability properties} we will show that, when $\s$ is the class of intrinsic Schreier split extensions, a stronger version of this lemma holds. Moreover, we will discuss the validity of other homological lemmas.

The following list of examples is given in chronological order; each one is a special case of the next one.
\begin{example}\cite{SchreierBook}
The category $\Mon$ of monoids is $\s$-protomodular with respect to the class $\s$ of Schreier split epimorphisms.
\end{example}

\begin{example}\cite{MartinsMontoli}
Every J\'onsson--Tarski variety is an $\s$-protomodular category with respect to the class $\s$ of Schreier split epimorphisms.
\end{example}

\begin{example}\cite{ise}\label{ex s-proto}
Every regular unital category with binary coproducts, equip\-ped with comonadic covers and a natural imaginary splitting, is $\s$-proto\-mo\-du\-lar with respect to the class $\s$ of intrinsic Schreier split epimorphisms. Consequently, such a split epimorphism, together with a chosen kernel, forms an \emph{intrinsic Schreier split extension}.
\end{example}

The reader may find several other examples in~\cite{BM-Nine-lemma}.

\section{The associativity axiom}\label{The associativity axiom}
In order to improve the behaviour of the intrinsic Schreier split extensions, it is useful to consider an additional assumption, concerning the associativity of the imaginary addition $\mu^X$.

Let $\C$ be a regular unital category with binary coproducts, comonadic covers and a natural imaginary splitting $t$. Suppose that, for every object $X$, the imaginary addition $\mu^X$ satisfies the associativity axiom $\mu^X\diam (\mu^X\times\overline{1_X}) = \mu^X\diam (\overline{1_X}\times\mu^X)$, i.e., the diagram
$$
\xymatrix@C=40pt@R=30pt{X\times X\times X \ar@{-->}[r]^-{\mu^X\times\overline{1_X}} \ar@{-->}[d]_-{\overline{1_X}\times\mu^X} & X\times X \ar@{-->}[d]^-{\mu^X} \\
 X\times X \ar@{-->}[r]_-{\mu^X} & X}
$$
commutes. This means that, for arbitrary imaginary morphisms $a$, $b$, $c\colon A\dashto X$, we get
\begin{equation}\label{assoc}
 \mu^X\diam \langle \mu^X\diam \langle a,b\rangle, c\rangle= \mu^X\diam \langle a, \mu^X\diam \langle b,c\rangle \rangle.
\end{equation}

\begin{example}
In $\Gp$ and in $\Mon$, the direct and twisted imaginary splittings induce associative imaginary additions.
\end{example}

Among the properties, listed in the previous section, of Schreier split epimorphisms of monoids, there is one, namely the property given by \S7, which uses the associativity of the monoid operation (for $X$). Hence it is not so surprising that we may prove its intrinsic version when we assume the associativity axiom.

\begin{proposition}\label{iS7} Suppose that the natural addition $(\mu^X\colon X\times X\dashto X)_{X\in \C}$ is associative. Given an intrinsic Schreier split extension \eqref{iSchreier} with imaginary retraction $q$, the following diagram commutes\\
\iS7 $\xymatrix@C=85pt@R=30pt{ X\times X \ar@{-->}[d]_-{\mu^X} \ar@{-->}[r]^-{\langle q\circ \pi_1, \overline{sf\pi_1}, k\circ q\circ \pi_2\rangle}
 & K\times X\times X \ar@{-->}[r]^-{\overline{1_K}\times (q\circ \mu^X)} & K\times K \ar@{-->}[d]^-{\mu^K} \\
 X \ar@{-->}[rr]_-q & & K}$\\ i.e., $\mu^K\circ (\overline{1_K}\times (q\circ \mu^X))\circ \langle q\circ \pi_1, \overline{sf\pi_1}, k\circ q\circ \pi_2\rangle = q\circ \mu^X$.
\end{proposition}
\begin{proof}
Using Lemma~\ref{aux} below, it suffices to prove that
\begin{align*}
 &\mu^X\circ \langle k\circ \mu^K\circ (\overline{1_K}\times (q\circ \mu^X))\circ \langle q\circ \pi_1, \overline{sf\pi_1}, k\circ q\circ \pi_2\rangle, \overline{sf}\circ \mu^X \rangle \\
 &= \mu^X \circ \langle k\circ q\circ \mu^X, \overline{sf}\circ \mu^X\rangle
\end{align*}
which, by $\iS1$, is the same as
\[
\mu^X\circ \langle k\circ \mu^K\circ (\overline{1_K}\times (q\circ \mu^X))\circ \langle q\circ \pi_1, \overline{sf\pi_1}, k\circ q\circ \pi_2\rangle, \overline{sf}\circ \mu^X \rangle = \mu^X.
\]
We have
\begin{align*}
 & \mu^X\circ \langle k\circ \mu^K\circ (\overline{1_K}\times (q\circ \mu^X))\circ \langle q\circ \pi_1, \overline{sf\pi_1}, k\circ q\circ \pi_2\rangle, \overline{sf}\circ \mu^X \rangle \\
& \underset{\qquad}{\stackrel{\eqref{naturality of mu}}{=}} \mu^X\circ \langle \mu^X \circ (k\times k)\circ \langle q\circ \pi_1, q\circ \mu^X \circ \langle \overline{sf\pi_1}, k\circ q\circ \pi_2 \rangle\rangle, \mu^X \circ (\overline{sf}\times \overline{sf}) \rangle \\
& \underset{\qquad}{=} \mu^X\circ \bigl\langle \mu^X \circ \bigl\langle k\circ q\circ \pi_1, k\circ q\circ \mu^X \circ \langle \overline{sf\pi_1}, k\circ q\circ \pi_2 \rangle\bigr\rangle, \mu^X \circ (\overline{sf}\times \overline{sf}) \bigr\rangle \\
& \underset{\qquad}{\stackrel{\eqref{assoc}}{=}} \mu^X\circ \bigl\langle k\circ q\circ \pi_1,\mu^X\circ \bigl\langle k\circ q\circ \mu^X \circ \langle \overline{sf\pi_1}, k\circ q\circ \pi_2 \rangle, \mu^X \circ (\overline{sf}\times \overline{sf})\bigr \rangle \bigr\rangle \\
& \underset{\qquad}{=} \mu^X\circ \Bigl\langle k\circ q\circ \pi_1,\mu^X\circ \bigl\langle k\circ q\circ \mu^X \circ \langle \overline{sf\pi_1}, k\circ q\circ \pi_2 \rangle, \mu^X \circ \bigl\langle \overline{sf\pi_1}, \overline{sf\pi_2}\bigr\rangle \bigr\rangle \Bigr\rangle \\
& \underset{\qquad}{\stackrel{\eqref{assoc}}{=}} \mu^X\circ \Bigl\langle k\circ q\circ \pi_1,\mu^X\circ \bigl\langle \mu^X \circ \bigl\langle k\circ q\circ \mu^X \circ \langle \overline{sf\pi_1}, k\circ q\circ \pi_2 \rangle, \overline{sf\pi_1}\bigr\rangle, \overline{sf\pi_2}\bigr\rangle \Bigr\rangle \\
& \underset{\qquad}{=} \mu^X\circ \langle k\circ q\circ \pi_1,\mu^X\circ \langle \mu^X \circ \langle k\circ q\circ \mu^X \circ (s\times k)\circ \langle \overline{f\pi_1}, q\circ \pi_2 \rangle,\\
 & \phantom{\underset{\qquad}{=} \mu^X\circ \langle k\circ q\circ \pi_1,\mu^X\circ \langle \mu^X \circ \langle}\overline{s\pi_Y}\circ \langle \overline{f\pi_1}, q\circ \pi_2\rangle\rangle , \overline{sf\pi_2}\rangle \rangle \\
 &\underset{\qquad}{=} \mu^X\circ \langle k\circ q\circ \pi_1,\\
&\phantom{\underset{\qquad}{=} \mu^X\circ \langle}\mu^X\circ \langle \mu^X \circ \langle k\circ q\circ \mu^X \circ (s\times k),\overline{s\pi_Y}\rangle \circ \langle \overline{f\pi_1}, q\circ \pi_2\rangle , \overline{sf\pi_2}\rangle \rangle \\
& \underset{\qquad}{\stackrel{\iS6}{=}} \mu^X\circ \langle k\circ q\circ \pi_1,\mu^X\circ \langle \mu^X \circ (s\times k) \circ \langle \overline{f\pi_1}, q\circ \pi_2\rangle , \overline{sf\pi_2}\rangle \rangle \\
& \underset{\qquad}{=} \mu^X\circ \langle k\circ q\circ \pi_1,\mu^X\circ \langle \mu^X \circ \langle \overline{sf\pi_1}, k\circ q\circ \pi_2\rangle , \overline{sf\pi_2}\rangle \rangle \\
& \underset{\qquad}{ \stackrel{\eqref{assoc}}{=}} \mu^X\circ \langle k\circ q\circ \pi_1,\mu^X\circ \langle \overline{sf\pi_1}, \mu^X\circ \langle k\circ q\circ \pi_2, \overline{sf\pi_2}\rangle \rangle \rangle \\
 &\underset{\qquad}{\stackrel{\iS1}{=}} \mu^X\circ \langle k\circ q\circ \pi_1,\mu^X\circ \langle \overline{sf\pi_1}, \overline{\pi_2} \rangle \rangle 
 \underset{\qquad}{ \stackrel{\eqref{assoc}}{=}} \mu^X\circ \langle \mu^X\circ \langle k\circ q\circ \pi_1,\overline{sf\pi_1} \rangle, \overline{\pi_2} \rangle \\
 &\underset{\qquad}{\stackrel{\iS1}{=}} \mu^X\circ \langle \overline{\pi_1},\overline{\pi_2}\rangle 
 \underset{\qquad}{=} \mu^X.\qedhere
\end{align*}
\end{proof}

\begin{lemma}\label{aux}Let \eqref{iSchreier} be a split epimorphism with an imaginary morphism $q$ such that \iS2 holds. If
$$ \mu^X \circ \langle k\circ a, s\circ b\rangle = \mu^X \circ \langle k\circ c, s\circ d\rangle,$$
where $a$, $c\colon A \dashto K$ and $b$, $d\colon A\dashto Y$ are imaginary morphisms, then $a=c$.
\end{lemma}
\begin{proof}
\begin{align*}
& \mu^X\circ \langle k\circ a, s\circ b\rangle = \mu^X \circ \langle k\circ c, s\circ d\rangle \\
& \stackrel{\phantom{\iS2}}{\Rightarrow}\; q\circ \mu^X \circ \langle k\circ a, s\circ b\rangle = q\circ \mu^X \circ \langle k\circ c, s\circ d\rangle \\
 &\stackrel{\phantom{\iS2}}{\Rightarrow}\; q\circ \mu^X \circ (k\times s) \circ \langle a,b\rangle = q\circ \mu^X \circ (k\times s)\circ \langle c, d\rangle \\
 &\stackrel{\iS2}{\Rightarrow}\; \overline{\pi_K}\circ \langle a,b\rangle = \overline{\pi_K}\circ \langle c, d\rangle \\
 &\stackrel{\phantom{\iS2}}{\Rightarrow}\; a=c.\qedhere
\end{align*}
\end{proof}

\begin{remark}
Of course in the situation of Lemma~\ref{aux}, we also have $b=d$---independently of whether \iS2 holds. Indeed, $f\circ\mu^X\circ \langle k\circ a, s\circ b\rangle =\mu^Y\circ \langle f\circ k\circ a, f\circ s\circ b\rangle=\mu^Y\circ \langle 0,b\rangle=b$, and by a similar argument we see that $f\circ\mu^X\circ \langle k\circ a, s\circ d\rangle =d$.
\end{remark}

\begin{remark}\label{q unique} As an immediate consequence of Lemma~\ref{aux}, we obtain the uniqueness of the imaginary retraction for any intrinsic Schreier split extension \eqref{iSchreier} (which was already known from~\cite[Proposition 5.3]{ise}). Given two possible imaginary retractions $q$, $q'\colon X\dashto K$, \iS1 gives
$$ \mu^X\circ \langle k\circ q, \overline{sf}\rangle = \overline{1_X} = \mu^X \circ \langle k\circ q', \overline{sf}\rangle; $$
consequently, $q=q'$.
\end{remark}

\begin{remark}
The referee suggested that we could replace the associativity condition by a weaker requirement adapted from Definition 1.4 in \cite{GJS}. We checked that this would indeed work for Proposition \ref{iS7}, but chose to keep the current, less general approach for the sake of simplicity.
\end{remark}

\section{Stability properties and homological lemmas}\label{Stability properties}
In this section we prove that certain stability properties for Schreier extensions of monoids shown in ~\cite{SchreierBook} still hold for intrinsic Schreier extensions in our context: $\C$ will denote a regular unital category with binary coproducts, comonadic covers and a natural imaginary splitting $t$. Moreover, we will observe that some of these stability properties allow to extend the validity of some classical homological lemmas to our intrinsic context.

Wherever we extend a proof in~\cite{SchreierBook} which uses the associativity of the monoid operations, we assume the associativity axiom holds. This is the case, in particular, of the first stability property we consider:

\begin{proposition}\emph{(See~\cite[Proposition~2.3.2]{SchreierBook})}\label{2.3.2, part 1} Suppose that the natural addition $(\mu^X\colon X\times X\dashto X)_{X\in \C}$ is associative. Then the intrinsic Schreier split extensions are stable under composition.
\end{proposition}
\begin{proof} Suppose that
$$
 \xymatrix@!0@=5em{ K \ar@<-.5ex>@{ |>->}[r]_-{k} & X \ar@<-.5ex>@{>>}[r]_-{f} \ar@<-.5ex>@{-->}[l]_-{q_f} & Y \ar@<-.5ex>[l]_-{s}}
$$
and
$$
 \xymatrix@!0@=5em{ L \ar@<-.5ex>@{ |>->}[r]_-{l} & Y \ar@<-.5ex>@{>>}[r]_-{g} \ar@<-.5ex>@{-->}[l]_-{q_g} & Z \ar@<-.5ex>[l]_-{t}}
$$
are intrinsic Schreier extensions. We want to prove that
$$
 \xymatrix@!0@=5em{ M \ar@{ |>->}[r]_-{m} & X \ar@<-.5ex>@{>>}[r]_-{gf} & Z \ar@<-.5ex>[l]_-{st}}
$$
is an intrinsic Schreier extension, where $m$ is the kernel of $gf$. Consider the following diagram where both squares are pullbacks:
$$
 \xymatrix@=30pt{ M \ar@{ |>->}[r]^-m \ar@<-.5ex>[d]_-{f'} \ophalfsplitpullback & X \ar@<-.5ex>[d]_-f \\
 L \ar@<-.5ex>[u]_-{s'} \ar@{ |>->}[r]^-l \ar[d] \pullback & Y \ar@<-.5ex>[u]_-{s} \ar@<-.5ex>[d]_-{g} \\
 0 \ar[r] & Z. \ar@<-.5ex>[u]_-{t} }
$$
We must provide an imaginary retraction $q\colon X\dashto M$. The imaginary morphism $\mu^X\circ \langle k\circ q_f, sl\circ q_g\circ f\rangle\colon X\dashto X$ is such that composing with $gf$ gives the following equalities in $\C$:
\begin{align*}
& gf\bi{1_X}{1_X}t_{X,X}P(\langle kq_f, slq_gP(f)\rangle) \delta_X \\
 &= gf\bi{1_X}{1_X}t_{X,X}P(kq_f\times slq_gP(f)) P(\langle 1_{P(X)},1_{P(X)}\rangle) \delta_X \\
 &\stackrel{\mathclap{\eqref{naturality of t}}}{=} gf\bi{1_X}{1_X}( kq_f + slq_gP(f) t_{P(X),P(X)}P(\langle 1_{P(X)},1_{P(X)}\rangle) \delta_X \\
 &= \bi{gfkq_f}{gfslq_gP(f)} t_{P(X),P(X)}P(\langle 1_{P(X)},1_{P(X)}\rangle) \delta_X 
 = 0.	
\end{align*}
This gives a unique morphism $q\colon P(X)\to M$ in $\C$, i.e., an imaginary morphism $q\colon X \dashto M$, such that $m\circ q= \mu^X\circ \langle k\circ q_f, sl\circ q_g\circ f\rangle$. We prove that this $q$ is the imaginary retraction for the split epimorphism $gf$. We start with \iS1 as in
\begin{align*}
 \mu^X \circ \langle m\circ q, \overline{stgf}\rangle 
 &= \mu^X \circ \langle \mu^X\circ \langle k\circ q_f, sl\circ q_g\circ f\rangle, \overline{stgf}\rangle \\
& \stackrel{\mathclap{\eqref{assoc}}}{=} \mu^X\circ \langle k\circ q_f, \mu^X\circ \langle sl\circ q_g\circ f, \overline{stgf} \rangle \rangle \\
 &= \mu^X\circ \langle k\circ q_f, \mu^X\circ (s\times s)\circ \langle l\circ q_g, \overline{tg} \rangle \circ f\rangle \\
 &\stackrel{\mathclap{\eqref{naturality of mu}}}{=} \mu^X\circ \langle k\circ q_f, s\circ \mu^Y\circ \langle l\circ q_g, \overline{tg} \rangle\circ f \rangle 
 = \overline{1_X},	
\end{align*}
where in the last step we use \iS1 for $g$ and then \iS1 for $f$.

Finally, we prove that $m\circ q\circ \mu^X\circ (m\times st) = m\circ \overline{\pi_M}$ and use the fact that $m$ is a monomorphism, to conclude \iS2. We start with
\begin{align*}
& m\circ q\circ \mu^X\circ (m\times st) \\
 &= \mu^X \circ \langle k\circ q_f, sl\circ q_g\circ f\rangle \circ \mu^X\circ (m\times st) \\
 &= \mu^X \circ \langle k\circ q_f \circ \mu^X\circ (m\times st), sl\circ q_g\circ f \circ \mu^X\circ (m\times st)\rangle \\
 &\stackrel{\mathclap{\eqref{naturality of mu}}}{=} \mu^X \circ \langle k\circ q_f \circ \mu^X\circ (m\times st), sl\circ q_g\circ \mu^Y\circ (f\times f)\circ (m\times st)\rangle \\
 &= \mu^X \circ \langle k\circ q_f \circ \mu^X\circ (m\times st), sl\circ q_g\circ \mu^Y\circ (fm\times fst)\rangle \\
 &= \mu^X \circ \langle k\circ q_f \circ \mu^X\circ (m\times st), sl\circ q_g\circ \mu^Y\circ (lf'\times t)\rangle \\
 &= \mu^X \circ \langle k\circ q_f \circ \mu^X\circ (m\times st), sl\circ q_g\circ \mu^Y\circ (l\times t) \circ (f'\times 1_Z)\rangle \\
 &= \mu^X \circ \langle k\circ q_f \circ \mu^X\circ (m\times st), sl\circ \overline{\pi_L}\circ (f'\times 1_Z)\rangle,	
\end{align*}
where in the last equality we use \iS2 for $g$. Now we use $sl\circ \overline{\pi_L}\circ (f'\times 1_Z)=\overline{sfm\pi_M}$ and \iS7 applied to $f$. This gives
\begin{align*}
 & m\circ q\circ \mu^X\circ (m\times st) \\
 &= \mu^X \circ \langle k\circ \mu^K\circ (\overline{1_K}\times (q_f\circ \mu^X))\circ \langle q_f\circ \pi_1, \overline{sf\pi_1},k\circ q_f\circ \pi_2\rangle 
 \circ(m\times st),\\
&\phantom{= \mu^X \circ \langle }\overline{sfm\pi_M} \rangle \\
 &= \mu^X \circ \langle k\circ \mu^K\circ \langle q_f\circ \pi_1\circ (m\times st), q_f\circ\mu^X\circ \langle \overline{sf\pi_1},k\circ q_f\circ \pi_2\rangle 
 \circ(m\times st)\rangle ,\\ 
&\phantom{= \mu^X \circ \langle }\overline{sfm\pi_M} \rangle.	
\end{align*}
Note that a part of the composite above is
\begin{align*}
 & q_f\circ \mu^X\circ \langle \overline{sf\pi_1}, k \circ q_f\circ \pi_2\rangle \circ (m\times st) \\
 &\stackrel{\mathclap{\eqref{naturality of mu}}}{=} \mu^K \circ (q_f\times q_f)\circ \langle \overline{sf\pi_1}, k\circ q_f\circ \pi_2\rangle \circ (m\times st) \\
 &= \mu^K \circ \langle q_f \circ \overline{sf\pi_1}\circ (m\times st), q_f\circ k\circ q_f\circ \pi_2 (m\times st) \rangle \\
& \stackrel{\mathclap{\iS3}}{=} \mu^K \circ \langle q_f \circ \overline{sf\pi_1}\circ (m\times st), q_f\circ st\pi_2 \rangle 
 \stackrel{\iS4}{=} 0.	
\end{align*}
Thus,
\begin{align*}
m\circ q\circ \mu^X\circ (m\times st) 
&\;\stackrel{\mathclap{\eqref{x+0=x}}}{=} \;\mu^X \circ \langle k\circ \mu^K\circ \langle 1_K,0\rangle\circ q_f\circ \overline{m\pi_M}, \overline{sfm\pi_M}\rangle \\
 &\;\stackrel{\mathclap{\eqref{x+0=x}}}{=} \;\mu^X\circ \langle k\circ q_f,\overline{sf}\rangle \circ \overline{m\pi_M}
 \\&\;\stackrel{\mathclap{\iS1}}{=}\; \overline{1_X}\circ \overline{m\pi_M} = m\circ \overline{\pi_M}.\qedhere
\end{align*}
\end{proof}

\begin{proposition}\emph{(See~\cite[Proposition~2.3.2]{SchreierBook})}\label{2.3.2, part 2} Consider split epimorphisms
$$
 \xymatrix@!0@=5em{X \ar@<-.5ex>[r]_-f & Y \ar@<-.5ex>[r]_-g \ar@<-.5ex>[l]_-s & Z \ar@<-.5ex>[l]_-t}
$$
in $\C$. If $(gf,st)$ is an intrinsic Schreier split extension, then so is $(g,t)$.
\end{proposition}
\begin{proof}
We use the same notation as in Proposition~\ref{2.3.2, part 1}. We claim that the needed imaginary retraction for $g$ is $q_g=f'\circ q \circ s\colon Y\dashto L$. From \iS1 for $gf$, we have
\begin{align*}
\mu^X\circ \langle m\circ q, \overline{stgf}\rangle = \overline{1_X} & \quad\Rightarrow\quad f \circ \mu^X\circ \langle m\circ q, \overline{stgf}\rangle = \overline{f} \\
 & \quad\stackrel{\mathclap{\eqref{naturality of mu}}}{\Rightarrow}\quad \mu^Y \circ \langle fm\circ q, \overline{tgf}\rangle = \overline{f}.
\end{align*}
Then
\begin{align*}
\mu^Y\circ \langle l\circ q_g,\overline{tg}\rangle & = \mu^Y \circ \langle lf'\circ q\circ s, \overline{tg}\rangle 
 = \mu^Y\circ \langle fm\circ q\circ s, \overline{tgfs}\rangle \\
 & = \mu^Y\circ \langle fm\circ q, \overline{tgf}\rangle\circ s 
 = \overline{f}\circ s 
 = \overline{1_Y},	
\end{align*}
which proves \iS1 for $g$. As for \iS2 for $g$, we have
\begin{align*}
q_g\circ \mu^Y\circ (l\times t) & = f'\circ q\circ s \circ \mu^Y\circ (l\times t) \\
 & \stackrel{\mathclap{\eqref{naturality of mu}}}{=} f'\circ q\circ \mu^X\circ (sl\times st) \\
 & = f'\circ q\circ \mu^X\circ (ms'\times st) \\
 & = f'\circ q\circ \mu^X\circ (m\times st) \circ (s'\times 1_Z) \\
 & = f'\circ \overline{\pi_M}\circ (s'\times 1_Z) 
 = f's'\circ \overline{\pi_L} 
 = \overline{\pi_L},	
\end{align*}
where we use \iS2 for $gf$ in the fifth equality.
\end{proof}

\begin{lemma}\label{aux2}Suppose that the values of $P$ are projective objects in $\C$. Let $a$, $b\colon A \dashto X$ be imaginary morphisms and $z\colon Z \twoheadrightarrow A$ a regular epimorphism. If $a\circ z=b\circ z$, then $a=b$.
\end{lemma}
\begin{proof} $a\circ z=b\circ z$ corresponds to the equality $aP(z)=bP(z)$ in $\C$. Then $a=b$, since $P(z)$ is a split epimorphism (see Remark~\ref{rem imaginary splitting}).
\end{proof}

In the following $\Eq(f)$ denotes the kernel pair of a morphism $f$.

\begin{proposition}\emph{(See~\cite[Proposition~2.3.5]{SchreierBook} and~\cite[Proposition 4.8]{BM-Nine-lemma}).}\label{2.3.5} Suppose that the values of $P$ are projective objects in $\C$. Consider the following commutative diagram
$$
 \xymatrix@=40pt{\Eq(\gamma) \ar@{ |>->}@<-.5ex>[r]_-{\kappa} \ar@<-1ex>[d]_-{\gamma_1} \ar@<1ex>[d]^-{\gamma_2} &
 \Eq(g) \ar@<-.5ex>@{-->}[l]_-{\rho} \ar@<-.5ex>@{>>}[r]_-{\varphi} \ar@<-1ex>[d]_-{g_1} \ar@<1ex>[d]^-{g_2} &
 \Eq(h) \ar@<-.5ex>[l]_-{\sigma} \ar@<-1ex>[d]_-{h_1} \ar@<1ex>[d]^-{h_2}\\
 K' \ar@{ |>->}@<-.5ex>[r]_-{k'} \ar[d]_-{\gamma} \ar[u] & X' \ar@<-.5ex>@{-->}[l]_-{q'} \ar@<-.5ex>@{>>}[r]_-{f'} \ar@{>>}[d]_-{g} \ar[u] & Y' \ar@<-.5ex>[l]_-{s'} \ar@{>>}[d]^-{h} \ar[u]\\
 K \ar@{ |>->}[r]_-{k} & X \ar@<-.5ex>@{>>}[r]_-f & Y. \ar@<-.5ex>[l]_-s}
$$
Note that, by the commutativity of limits, $\kappa$ is the kernel of $\varphi$. If the top two rows are intrinsic Schreier split extensions and $g$ and $h$ are regular epimorphisms, then the bottom row is also an intrinsic Schreier split extension.
\end{proposition}
\begin{proof} $\C$ is an $\s$-protomodular category (Example~\ref{ex s-proto}), thus it is an $\s$-Mal'tsev category~\cite[Theorem 5.4]{BM-Nine-lemma}. By Proposition 3.2 in~\cite{BM-Nine-lemma}, $\gamma$ is a regular epimorphism.

Since $P(X)$ is a projective object and $g$ is a regular epimorphism, $g$ admits an imaginary splitting $t\colon X\dashto X'$, so that $g\circ t=\overline{1_X}$ (see Remark~\ref{rem imaginary splitting}). We claim that $q=\gamma\circ q'\circ t\colon X\dashto K$ is the needed imaginary retraction for the bottom row. We must prove \iS1:
\begin{align*}
\mu^X\circ \langle k\circ q, \overline{sf}\rangle & = \mu^X\circ \langle k\gamma\circ q'\circ t, \overline{sf}\rangle \\
 & = \mu^X\circ \langle gk'\circ q'\circ t, \overline{sf}\circ g\circ t\rangle \\
 & = \mu^X\circ \langle gk'\circ q'\circ t, g\circ \overline{s'f'}\circ t\rangle \\
 & = \mu^X\circ (g\times g) \circ \langle k'\circ q', \overline{s'f'}\rangle \circ t \\
 & \stackrel{\mathclap{\eqref{naturality of mu}}}{=} g\circ \mu^{X'} \circ \langle k'\circ q', \overline{s'f'}\rangle \circ t \\
 & = g\circ \overline{1_{X'}} \circ t
 = \overline{1_X},	
\end{align*}
where we use \iS1 applied to the second row in the next to last equality.

For \iS2, we precompose the equality we wish to prove with the regular epimorphism $\gamma\times h$
\begin{align*}
q\circ \mu^X\circ (k\times s)(\gamma\times h) & = \gamma\circ q'\circ t\circ \mu^X\circ (k\gamma\times sh) \\
 & = \gamma\circ q'\circ t\circ \mu^X\circ (gk'\times gs') \\
 & = \gamma\circ q'\circ t\circ \mu^X\circ (g\times g)(k'\times s') \\
 &\stackrel{\mathclap{\eqref{naturality of mu}}}{=} \gamma\circ q'\circ t\circ g\circ \mu^{X'}\circ (k'\times s') \\
 & \stackrel{\mathclap{(*)}}{=} \gamma\circ q'\circ \mu^{X'}\circ (k'\times s') = \gamma\circ \overline{\pi_{K'}} 
 = \overline{\pi_K}\circ (\gamma\times h),	
\end{align*}
where we use \iS2 for the second row in the next to last equality. Then \iS2 follows from Lemma~\ref{aux2}.

To finish, we just need to prove the equality $\gamma \circ q'\circ t\circ g\stackrel{(*)}{=}\gamma \circ q'$. Actually, we prove this equality in $\C$ (not in $\K$) and to do so, we use the compatibility of the first two rows with respect to the imaginary retractions \cite[Proposition 5.7]{ise}: $\gamma_i \rho=q'P(g_i)$, for $i\in \{1,2\}$. We have
\begin{align*}
\gamma q' P(t)\delta_XP(g) & \stackrel{\eqref{naturality of delta}}{=} \gamma q'P(t) P^2(g)\delta_{X'}
 = \gamma q'P(tP(g))\delta_{X'} 
 = \gamma q'P(g_1\langle tP(g), \varepsilon_{X'}\rangle) \delta_{X'},	
\end{align*}
where $\langle tP(g), \varepsilon_{X'}\rangle$ is the unique morphism making the following diagram commutative
$$\xymatrix@=30pt{ P(X') \ar@{.>}[dr]|-{\langle tP(g), \varepsilon_{X'}\rangle} \ar@(r,u)[drr]^-{\varepsilon_{X'}} \ar@(d,l)[ddr]_-{tP(g)} \\
 & \Eq(g) \pullback \ar[d]_-{g_1} \ar[r]^-{g_2} & X' \ar@{>>}[d]^-g \\
 & X' \ar@{>>}[r]_-g & X.}
$$
Using the compatibility mentioned earlier,
\begin{align*}
\gamma q' P(t)\delta_XP(g) \stackrel{\phantom{\eqref{counit}}}{=} & \gamma q'P(g_1)P(\langle tP(g), \varepsilon_{X'}\rangle) \delta_{X'} \\
 \stackrel{\phantom{\eqref{counit}}}{=} & \gamma \gamma_1 \rho P(\langle tP(g), \varepsilon_{X'}\rangle) \delta_{X'} \\
 \stackrel{\phantom{\eqref{counit}}}{=} & \gamma \gamma_2 \rho P(\langle tP(g), \varepsilon_{X'}\rangle) \delta_{X'} \\
 \stackrel{\phantom{\eqref{counit}}}{=} & \gamma q'P(g_2)P(\langle tP(g), \varepsilon_{X'}\rangle) \delta_{X'} \\
 \stackrel{\phantom{\eqref{counit}}}{=} & \gamma q'P(\varepsilon_{X'}) \delta_{X'} 
 \stackrel{\eqref{counit}}{=} \gamma q'.\qedhere
\end{align*}
\end{proof}

\begin{corollary}\emph{(See~\cite[Corollary~2.3.6]{SchreierBook})}\label{2.3.6} Suppose that the values of $P$ are projective objects in $\C$. Consider the diagram
$$
\xymatrix@=40pt{K \ar@{ |>->}[r]^-{k'} \ar@{=}[d] & X' \ar@{>>}@<-.5ex>[r]_-{f'} \ar@{>>}[d]_-{g} \halfsplitpullback & Y' \ar@<-.5ex>[l]_-{s'} \ar@{>>}[d]^-{h}\\
 K \ar@{ |>->}[r]^-{k} & X \ar@{>>}@<-.5ex>[r]_-{f} & Y, \ar@<-.5ex>[l]_-{s}}
$$
where the three squares involving, respectively, the split epimorphism, the kernels, and the sections commute. If the top row is an intrinsic Schreier split extension, then so is the bottom row.
\end{corollary}
\begin{proof}
Take the kernel pairs of the regular epimorphisms $1_K$, $g$ and $h$. This gives a $3 \times 3$ diagram whose top row is an intrinsic Schreier split extension, since these extensions are closed under arbitrary pullbacks (see~\cite[Proposition~6.1]{ise}). Applying Proposition~\ref{2.3.5} to this $3\times 3$ diagram, we conclude that $(f,s)$ is an intrinsic Schreier split extension.
\end{proof}

In order to get the validity, in our context, of one of the classical homological lemmas, namely the $3 \times 3$-Lemma, we need another stability property of the class of intrinsic Schreier split epimorphisms. This property was called \defn{equi-consistency} in \cite{BM-Nine-lemma}:

\begin{definition} \cite[Definition $6.3$]{BM-Nine-lemma}
Let $\s$ be a pullback-stable class of points. Consider any commutative diagram as in Figure~\ref{3x3}, 
\begin{figure}
$\xymatrix@=40pt{
K'' \ar@{ |>->}[r]_-{k''} \ar@{.>}[d]_-{w} &
X'' \ar@{>>}@<-.5ex>[r]_-{f''} \ar@{ |>->}@<-.5ex>[d]_-{u} &
Y'' \ar@<-.5ex>[l]_-{s''} \ar@{ |>->}[d]_-{v}\\
K' \ar@{ |>->}[r]_-{k'} \ar@{.>}@<-.9ex>[d]_-{t_1} \ar@{.>}@<+.9ex>[d]^-{t_2} &
R \ar@{>>}@<-.5ex>[r]_-{f'} \ar@<-1ex>[d]_-{r_1} \ar@<+1ex>[d]^-{r_2} &
S \ar@<-.5ex>[l]_-{s'} \ar@<-1ex>[d]_-{s_1} \ar@<+1ex>[d]^-{s_2} \\
K \ar@{ |>->}[r]_-{k} \ar@{.>}[u] & X \ar@{>>}@<-.5ex>[r]_-{f} \ar[u]|{e_R} &
Y \ar@<-.5ex>[l]_-{s} \ar[u]|{e_S}
}$
\caption{Equi-consistency}\label{3x3}
\end{figure}
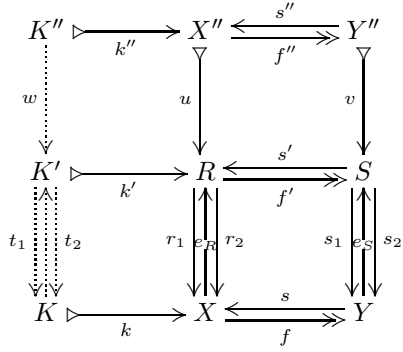
where $\langle r_1,r_2\rangle\colon R \rightarrowtail X\times X$ and $\langle s_1,s_2\rangle\colon S\rightarrowtail Y\times Y$ are equivalence relations, $(f,s)$ and $(f',s')$ are split epimorphisms, $(f'',s'')$ is the induced split epimorphism between the kernels of $r_1$ and $s_1$, and the diagram is completed by taking kernels and the induced dotted morphisms. $\s$ is \defn{equi-consistent} if, whenever the points $(f,s)$, $(r_1,e_R)$, and $(f'',s'')$ belong to $\s$, then $(f',s')$ is in $\s$, too.
\end{definition}

\begin{proposition}\emph{(See~\cite[Proposition~6.4]{BM-Nine-lemma})}\label{Prop 6.4} Suppose that the natural addition $(\mu^X\colon X\times X\dashto X)_{X\in \C}$ is associative. Then the class of intrinsic Schreier split extensions is equi-consistent.
\end{proposition}

\begin{proof} Consider the commutative diagram in Figure~\ref{3x3}. Suppose that $(f,s)$, $(r_1,e_R)$, and $(f'',s'')$ are intrinsic Schreier split epimorphisms, and denote by $q$, $q''$, and $\chi$ the imaginary retractions for $(f,s)$, $(f'',s'')$, and $(r_1,e_R)$, respectively. In particular, we have
\begin{equation}\label{retraction pp}
t_iw\circ q'' = q\circ r_iu, \; i\in\{1,2\}.
\end{equation}

We consider the imaginary morphism $\alpha=\mu^R\circ \langle s'f'u\circ \chi, e_Rk\circ q\circ r_1\rangle\colon R\dashto R$. We have
\begin{align*}
& r_1\circ \mu^R\circ \langle s'f'u\circ \chi, e_Rk\circ q\circ r_1\rangle 
\stackrel{\eqref{naturality of mu}}{=} \mu^X\circ (r_1\times r_1)\circ \langle s'f'u\circ \chi, e_Rk\circ q\circ r_1\rangle \\
&\stackrel{\text{Fig.~\ref{3x3}}}{=} \mu^X\circ \langle 0, k\circ q\circ r_1\rangle 
= \mu^X\circ \langle 0,1_X\rangle \circ k\circ q\circ r_1 
\stackrel{\eqref{0+x=x}}{=} k\circ q\circ r_1	
\end{align*}

and
\begin{align*}
& r_2\circ \mu^R\circ \langle s'f'u\circ \chi, e_Rk\circ q\circ r_1\rangle \\
&\stackrel{\mathclap{\eqref{naturality of mu}}}{=} \quad \mu^X\circ (r_2\times r_2)\circ \langle s'f'u\circ \chi, e_Rk\circ q\circ r_1\rangle \\
&\stackrel{\mathclap{\text{Fig.~\ref{3x3}}}}{=} \quad \mu^X\circ \langle sfr_2u\circ \chi, k\circ q\circ r_1\rangle;	
\end{align*}
thus we obtain the commutativity of the following diagram:
\begin{equation}\label{R1}
\vcenter{\xymatrix@C=60pt{
& X \\
R \ar@{-->}[r]^-{\alpha} \ar@{-->}[ur]^-{k\circ q\circ r_1} \ar@{-->}[dr]_-{\mu^X\circ \langle sfr_2u\circ \chi, k\circ q\circ r_1\rangle\;\;\;\;\;\;} &
R \ar[u]_-{r_1} \ar[d]^-{r_2} \\
& X.
}}
\end{equation}

Next, we consider $\beta=\mu^R\circ \langle e_Rk\circ q\circ r_2\circ \alpha, s'f'u\circ \chi\rangle\colon R\dashto R$. We have
\begin{align*}
& r_1\circ \mu^R\circ \langle e_Rk\circ q\circ r_2\circ \alpha, s'f'u\circ \chi\rangle \\
&\stackrel{\mathclap{\eqref{naturality of mu}}}{=} \quad\mu^X\circ (r_1\times r_1)\circ \langle e_Rk\circ q\circ r_2\circ \alpha, s'f'u\circ \chi\rangle \\
&\stackrel{\mathclap{\text{Fig.~\ref{3x3}}}}{=} \quad \mu^X\circ \langle k\circ q\circ r_2\circ \alpha,0\rangle \\
&= \quad \mu^X\circ \langle 1_X,0\rangle \circ k\circ q\circ r_2\circ \alpha \\
&\stackrel{\mathclap{\eqref{x+0=x}}}{=} \quad k\circ q\circ r_2\circ \alpha	
\end{align*}
and
\begin{align*}
& r_2\circ \mu^R\circ \langle e_Rk\circ q\circ r_2\circ \alpha, s'f'u\circ \chi\rangle \\
&\stackrel{\mathclap{\eqref{naturality of mu}}}{=} \quad \mu^X\circ (r_2\times r_2)\circ \langle e_Rk\circ q\circ r_2\circ \alpha, s'f'u\circ \chi\rangle \\
&\stackrel{\mathclap{\text{\eqref{R1},Fig.~\ref{3x3}}}}{=} \quad\mu^X\circ \langle k\circ q\circ \mu^X\circ \langle sfr_2u\circ \chi, k\circ q\circ r_1\rangle, sfr_2u\circ \chi \rangle \\
&= \quad \mu^X\circ \langle k\circ q\circ \mu^X\circ (s\times k)\circ \langle fr_2u\circ \chi, q\circ r_1\rangle, sfr_2u\circ \chi \rangle \\
&= \quad \mu^X\circ \langle k\circ q\circ \mu^X\circ (s\times k)\circ \langle fr_2u\circ \chi, q\circ r_1\rangle, s\pi_Y\circ \langle fr_2u\circ \chi, q\circ r_1\rangle \rangle \\
&= \quad \mu^X\circ \langle k\circ q\circ \mu^X\circ (s\times k), \overline{s\pi_Y} \rangle \circ \langle fr_2u\circ \chi, q\circ r_1\rangle \\
&\stackrel{\mathclap{\iS6}}{=}\quad \mu^X\circ (s\times k)\circ \langle fr_2u\circ \chi, q\circ r_1\rangle
= \mu^X\circ \langle sfr_2u\circ \chi, k\circ q\circ r_1\rangle;
\end{align*}
this gives the commutativity of the following diagram:
\begin{equation}\label{R2}
\vcenter{\xymatrix@C=50pt{
& & & X \\
R \ar@{-->}[rr]|-{\beta} \ar@{-->}[urrr]^-{\mu^X\circ \langle sfr_2u\circ \chi, k\circ q\circ r_1\rangle\;\;\;\;\;\;\;\;} \ar@{-->}[drrr]_-{k\circ q\circ r_2\circ \alpha} & &
R \ar[ur]_-{r_2} \ar[dr]^-{r_1} \ar[r]|-{i_R} & R \ar[u]_-{r_1} \ar[d]^-{r_2} \\
& & & X.
}}
\end{equation}

Now we consider $\gamma=\mu^R\circ \langle uk''\circ q''\circ\chi, e_Rk\circ q\circ r_2\circ \alpha\rangle\colon R\dashto R$. We have
\begin{align*}
& r_1\circ \mu^R\circ \langle uk''\circ q''\circ\chi, e_Rk\circ q\circ r_2\circ \alpha\rangle \\
&\stackrel{\mathclap{\eqref{naturality of mu}}}{=} \quad \mu^X\circ (r_1\times r_1)\circ \langle uk''\circ q''\circ\chi, e_Rk\circ q\circ r_2\circ \alpha\rangle \\
&\stackrel{\mathclap{\text{Fig.~\ref{3x3}}}}{=} \quad \mu^X\circ \langle 0, k\circ q\circ r_2\circ \alpha \rangle \\
&= \quad \mu^X\circ \langle 0,1_X\rangle \circ k\circ q\circ r_2\circ \alpha \\
&\stackrel{\mathclap{\eqref{0+x=x}}}{=} \quad k\circ q\circ r_2\circ \alpha
\end{align*}
and
\begin{align*}
& r_2\circ \mu^R\circ \langle uk''\circ q''\circ\chi, e_Rk\circ q\circ r_2\circ \alpha\rangle \\
&\stackrel{\mathclap{\eqref{naturality of mu}}}{=} \qquad \mu^X\circ (r_2\times r_2)\circ \langle uk''\circ q''\circ\chi, e_Rk\circ q\circ r_2\circ \alpha\rangle \\
&\stackrel{\mathclap{\text{\eqref{R1},Fig.~\ref{3x3}}}}{=} \qquad \mu^X\circ \langle kt_2w\circ q''\circ \chi, k\circ q\circ \mu^X\circ \langle sfr_2u\circ \chi,k\circ q\circ r_1\rangle \rangle \\
&= \qquad \mu^X\circ (k\times k) \circ \langle t_2w\circ q''\circ \chi, q\circ \mu^X\circ \langle sfr_2u\circ \chi,k\circ q\circ r_1\rangle \rangle \\
&\stackrel{\mathclap{\eqref{retraction pp}}}{=} \qquad \mu^X\circ (k\times k) \circ \langle q\circ r_2u\circ \chi, q\circ \mu^X\circ \langle sfr_2u\circ \chi,k\circ q\circ r_1\rangle \rangle \\
&\stackrel{\mathclap{\eqref{naturality of mu}}}{=} \qquad k\circ \mu^K\circ (\overline{1_K}\times (q\circ \mu^X)) \circ \langle q\circ r_2u\circ \chi, \langle sfr_2u\circ \chi,k\circ q\circ r_1\rangle \rangle \\
&= \qquad k\circ \mu^K\circ (\overline{1_K}\times (q\circ \mu^X)) \circ \langle q\circ \pi_1\circ \langle r_2u\circ \chi, \overline{r_1}\rangle,\\
&\phantom{= \qquad k\circ \mu^K\circ (\overline{1_K}\times (q\circ \mu^X)) \circ \langle} \langle \overline{sf\pi_1} \circ \langle r_2u\circ \chi, \overline{r_1}\rangle , k\circ q\circ\pi_2\circ \langle r_2u\circ \chi, \overline{r_1}\rangle \rangle \rangle \\
&= \qquad k\circ \mu^K\circ (\overline{1_K}\times (q\circ \mu^X)) \circ \langle q\circ \pi_1, \langle \overline{sf\pi_1}, k\circ q\circ \pi_2\rangle \rangle \circ \langle r_2u\circ \chi, \overline{r_1} \rangle \\
&\stackrel{\mathclap{\iS7}}{=} \qquad k\circ q\circ \mu^X\circ \langle r_2u\circ \chi, \overline{r_1} \rangle \\
&= \qquad k\circ q\circ \mu^X \circ \langle r_2u\circ \chi, r_2\circ \overline{e_R r_1} \rangle \\
&= \qquad k\circ q\circ \mu^X \circ (r_2\times r_2) \circ \langle u\circ \chi, \overline{e_R r_1} \rangle \\
&\stackrel{\mathclap{\eqref{naturality of mu}}}{=} \qquad k\circ q\circ r_2\circ \mu^R \circ \langle u\circ \chi, \overline{e_R r_1} \rangle \\
&\stackrel{\mathclap{\iS1}}{=} \qquad k\circ q\circ r_2,
\end{align*}
where the last \iS1 is with respect to the intrinsic Schreier extension $(r_1,e_R)$; then the diagram
\begin{equation}\label{R3}
\vcenter{\xymatrix@C=60pt{
& X \\
R \ar@{-->}[r]^-{\gamma} \ar@{-->}[ur]^-{k\circ q\circ r_2\circ \alpha} \ar@{-->}[dr]_-{k\circ q\circ r_2} &
R \ar[u]_-{r_1} \ar[d]^-{r_2} \\
& X
}}
\end{equation}
commutes.

Next we use the fact that $R$ is transitive together with \eqref{R1}, \eqref{R2} and \eqref{R3} to deduce the existence of an imaginary morphism $\delta\colon R\dashto R$ such that the following diagram commutes
$$
\xymatrix@C=60pt{
& X \\
R \ar@{-->}[r]^-{\delta} \ar@{-->}[ur]^-{k\circ q\circ r_1} \ar@{-->}[dr]_-{k\circ q\circ r_2} &
R \ar[u]_-{r_1} \ar[d]^-{r_2} \\
& X.
}
$$
We are now able to define the imaginary retraction $q'$ for $(f',s')$:
$$
\xymatrix@=30pt{
R \ar@{-->}[dr]|-{q'} \ar@{-->}@(r,u)[drr]^-{\delta} \ar@{-->}@(d,l)[ddr]_-{!_R} \\
 & K' \pullback \ar[d]_-{!_{K'}} \ar[r]^-{k'} & R \ar@{>>}[d]^-{f'} \\
 & 0 \ar@{>>}[r] & S.
}
$$
Note that $s_if'\circ \delta =fr_i\circ \delta=fk\circ q\circ r_i=0$, $i\in\{1,2\}$, from which we get $f'\circ \delta=0$.

To finish we must prove \iS1 and \iS2 for $(f',s')$. To obtain the equality for \iS1
$$
\xymatrix@C=40pt@R=30pt{R \ar@{-->}[r]^-{\langle k'\ci q', \overline{s'f'}\rangle} \ar@{-->}[dr]_-{\overline{1_R}} & R\times R \ar@{-->}[d]^-{\mu^R} \\ & R,}
$$
we prove that $r_i\circ \mu^R \circ \langle k'\circ q', \overline{s'f'}\rangle = r_i\circ \overline{1_R}=\overline{r_i}$, $i\in\{1,2\}$, by using \iS1 for~$(f,s)$.

To obtain the equality for \iS2
$$
\xymatrix@=30pt{ K'\times S \ar[r]^-{k'\times s'} \ar@{-->}[drr]_-{\overline{\pi_{K'}}} & R\times R \ar@{-->}[r]^-{\mu^R} & R \ar@{-->}[d]^-{q'} \\ & & K',}
$$
we prove that $r_i\circ k'\circ q'\circ \mu^R\circ (k'\times s')=r_i\circ k'\circ \overline{\pi_{K'}}$, $i\in\{1,2\}$, which uses \iS2 for $(f,s)$.
\end{proof}

We say that a morphism $f \colon X \to Y$ is an \defn{intrinsic Schreier special morphism} if the split epimorphism $(f_1, e_f)$ is intrinsic Schreier, where $\xymatrix{ \Eq(f) \ar@<-1ex>[r]_-{f_1} \ar@<+1ex>[r]^-{f_2} & X \ar[l]|-{e_f} }$ is the kernel pair of $f$. This is equivalent to asking that the split epimorphism $(f_2, e_f)$ is intrinsic Schreier. If this intrinsic Schreier special morphism is a regular epimorphism, then it is automatically the cokernel of its kernel, thus it gives rise to an extension (Proposition 5.6 in~\cite{BM-Nine-lemma}). Thanks to the stability properties we proved in this section, we can apply Proposition~6.2 and Theorem~6.7 in \cite{BM-Nine-lemma} and get the following version of the $3 \times 3$-Lemma:

\begin{theorem}
Consider the commutative diagram
$$
\vcenter{\xymatrix@=40pt{
K'' \ar[r]^-{k''} \ar@{ |>->}[d]_-{w'} &
X'' \ar[r]^-{f''} \ar@{ |>->}[d]_-{u'} &
Y'' \ar@{ |>->}[d]^-{v'}\\
K' \ar@{ |>->}[r]_-{k'} \ar@{>>}[d]_-{w} &
X' \ar@{>>}[r]_-{f'} \ar@{>>}[d]_-{u} &
Y' \ar@{>>}[d]^-{v} \\
K \ar[r]_-{k} \ar@{.>}[u] & X \ar[r]_-{f} &
Y,
}}
$$
where the three columns and the middle row are intrinsic Schreier special extensions. The upper row is an intrinsic Schreier special extension if and only if the lower one is.\noproof
\end{theorem}

We conclude this section by proving the stronger version of the Split Short Five Lemma we mentioned in Section \ref{Intrinsic Schreier split epimorphisms}.

\begin{proposition}\emph{(See~\cite[Proposition~2.3.10]{SchreierBook})}\label{2.3.10} Suppose that the values of $P$ are projective objects in $\C$. Consider the diagram
$$
\xymatrix@=40pt{K' \ar@{ |>->}@<-.5ex>[r]_-{k'} \ar[d]_-{\gamma} & X' \ar@{>>}@<-.5ex>[r]_-{f'} \ar[d]_-{g} \ar@{-->}@<-.5ex>[l]_-{q'}
 & Y' \ar@<-.5ex>[l]_-{s'} \ar[d]^-{h}\\
 K \ar@{ |>->}@<-.5ex>[r]_-{k} & X \ar@{>>}@<-.5ex>[r]_-{f} \ar@{-->}@<-.5ex>[l]_-{q} & Y, \ar@<-.5ex>[l]_-{s}}
$$
where both rows are intrinsic Schreier split extensions and the three squares involving, respectively, the split epimorphism, the kernels, and the sections commute. Then
\begin{enumerate}
\item $g$ is a regular epimorphism if and only if $\gamma$ and $h$ are regular epimorphisms;
\item $g$ is a monomorphism if and only if $\gamma$ and $h$ are monomorphisms.
\end{enumerate}
\end{proposition}
\begin{proof}
1.\ If $g$ is a regular epimorphism, then so is $h$, from the commutativity of the diagram. Moreover, the compatibility for the imaginary retractions gives $\gamma q'=qP(g)$. Then $\gamma$ is a regular epimorphism since so are $q$ and $P(g)$ (by \iS2 and Remark~\ref{rem imaginary splitting}).

Conversely, suppose that $\gamma$ and $h$ are regular epimorphisms. We take the (regular epimorphisms, monomorphism) factorisation $g=me$, and prove that $m$ is an isomorphism. Since the bottom row is an intrinsic Schreier split extension, we know that $(k,s)$ is a jointly extremal-epimorphic pair (see Subsection~\ref{ISSE}). Since $q$ and $h\varepsilon_{Y'}$ are regular epimorphisms, then $(kq,sh\varepsilon_{Y'})$ is also a jointly extremal-epimorphic pair. In $\C$, it is easy to check the commutativity of
$$
\xymatrix@=30pt{& M \ar@{ >->}[d]_-m \\
 P(X) \ar[r]_-{kq} \ar[ur]^-{\varepsilon_MP(ek')\sigma P(q)\delta_X} & X & P(Y'), \ar[l]^-{sh\varepsilon_{Y'}} \ar[ul]_-{\varepsilon_M P(es')} }
$$
where $\sigma$ is a splitting of the split epimorphism $P(\gamma)$ (Remark~\ref{rem imaginary splitting}). Thus, $m$ is an isomorphism and $g$ is a regular epimorphism.

\smallskip
2.\ If $g$ is a monomorphism, then so are $\gamma$ and $h$, from the commutativity of the diagram. For the converse, suppose that $a,b\colon U\to X'$ are morphisms such that $ga=gb$. Then, $fga=fgb$, from which we get $f'a=f'b$ (since $fg=hf'$ and $h$ is a monomorphism). On the other hand, we deduce $kqP(g)P(a)=kqP(g)P(b)$ and $k\gamma q'P(a)=k\gamma q'P(b)$, from the compatibility for imaginary retractions (\cite[Proposition 5.7]{ise}). This gives $q'P(a)=q'P(b)$, since $k\gamma$ is a monomorphism. Thus $q'\circ a=q'\circ b$, as imaginary morphisms. Then
\begin{align*}
a &= \overline{1_{X'}}\circ a 
 \stackrel{\iS1}{=} \mu^{X'} \circ \langle k'\circ q',\overline{s'f'}\rangle \circ a 
 = \mu^{X'} \circ \langle k'\circ q'\circ a,\overline{s'f'a}\rangle \\
 &= \mu^{X'} \circ \langle k'\circ q'\circ b,\overline{s'f'b}\rangle 
 = \mu^{X'} \circ \langle k'\circ q',\overline{s'f'}\rangle \circ b 
 \stackrel{\iS1}{=} \overline{1_{X'}} \circ b 
 = b.\qedhere
\end{align*}
\end{proof}

\section{Intrinsic Schreier special objects}\label{Intrinsic Schreier special objects}
Let $\C$ be a pointed and finitely complete category and $\s$ a class of points in~$\C$ which is stable under pullbacks along arbitrary morphisms. Recall from~\cite{S-proto} that an object $Y$ is called an \defn{$\s$-special object} when the split epimorphism
\begin{equation} \label{point S-special object}
\xymatrix@!0@=6em{ Y \ar@{ |>->}[r]_-{\langle 1_Y,0 \rangle} & Y\times Y \ar@<-.5ex>@{>>}[r]_-{\pi_2} & Y \ar@<-.5ex>[l]_-{\langle 1_Y,1_Y \rangle}}
\end{equation}
(or, equivalently, the split epimorphism $(\pi_1,\langle 1_Y,1_Y\rangle)$) belongs to the class~$\s$. If $\C$ is an $\s$-protomodular category, then the full subcategory formed by the $\s$-special objects is protomodular (\cite{S-proto}, Proposition $6.2$), and it is called the \defn{protomodular core} of $\C$ with respect to the class $\s$. When $\C$ is the category of monoids, and $\s$ is either the class of Schreier split epimorphisms or the one of left homogeneous split epimorphisms, the protomodular core is the category of groups. More generally, when $\V$ is a J\'onsson--Tarski variety, an algebra in $\V$ is a Schreier special object if and only if it has a right loop structure~\cite[Proposition 7.5]{ise} (see Subsection~\ref{Imaginary (one-sided) loops} for the right loop axioms). Similarly, an algebra in $\V$ is special with respect to the class of left homogeneous split epimorphisms (see Remark~\ref{homogeneous}) if and only if it has a left loop structure (see Subsection~\ref{Imaginary (one-sided) loops} for the left loop axioms).

Now we want to study what happens in the intrinsic Schreier setting. So, let~$\C$ be a regular unital category with binary coproducts, comonadic covers and a natural imaginary splitting $t$. An object $Y$ in $\C$ is an \defn{intrinsic Schreier special object} when the split epimorphism \eqref{point S-special object} is an intrinsic Schreier split epimorphism. This means that there exists an imaginary morphism $q\colon Y\times Y \dashto Y$ such that:
\begin{itemize}
\item[\iSs1] the diagram
$$
\xymatrix@C=110pt@R=30pt{Y\times Y \ar@{-->}[r]^-{\langle \langle 1_Y,0\rangle \ci q, \overline{\langle 1_Y,1_Y\rangle \pi_2}\rangle} \ar@{-->}[dr]_-{\overline{1_{Y\times Y}}} & Y\times Y \times Y\times Y \ar@{-->}[d]^-{\mu^{Y\times Y}} \\ & Y\times Y}
$$
commutes;
\item[\iSs2] the diagram
$$
\xymatrix@C=80pt@R=30pt{ Y\times Y \ar[r]^-{\langle 1_Y,0\rangle\times\langle 1_Y,1_Y\rangle} \ar@{-->}[drr]_-{\overline{\pi_1}} & Y\times Y\times Y\times Y \ar@{-->}[r]^-{\mu^{Y\times Y}} & Y\times Y \ar@{-->}[d]^-q \\ & & Y}
$$
 commutes.
\end{itemize}

In this context, we also have:
\begin{itemize}
\item[\iSs3] $\xymatrix@=30pt{Y \ar[r]^-{\langle 1_Y,0\rangle} \ar@{-->}@/_1pc/[rr]_-{q\ci \langle 1_Y,0\rangle=\overline{1_Y}} & Y\times Y \ar@{-->}[r]^-{q} & Y,}$ i.e., $qP(\langle 1_Y,0\rangle)=\varepsilon_Y$;
\item[\iSs4] $\xymatrix@=30pt{Y \ar[r]^-{\langle 1_Y,1_Y\rangle} \ar@{-->}@/_1pc/[rr]_-{q\ci \langle 1_Y,1_Y\rangle=\overline{0}} & Y\times Y \ar@{-->}[r]^-{q} & Y,}$ i.e., $qP(\langle 1_Y,1_Y\rangle) = 0$.
\end{itemize}
So, if $Y$ is an intrinsic Schreier special object, then the identities \iSs3 and \iSs4 make $q\colon Y\times Y\dashto Y$ an \defn{imaginary subtraction}. Indeed, the identities \iSs3 and \iSs4 correspond to the varietal axioms for a subtraction, i.e.,
\[ q(x, 0) = x, \qquad q(x, x) = 0. \]

\subsection{Imaginary (one-sided) loops}\label{Imaginary (one-sided) loops}
Consider an intrinsic Schreier special object~$Y$. We now show that the imaginary addition given in~\eqref{imaginary addition} and the imaginary subtraction $q\colon Y\times Y\dashto Y$ satisfy the axioms of a (one-sided) loop (like those of a right loop or a left loop). We say then that $Y$ has the structure of an \defn{imaginary one-sided loop}.

We must prove the right loop or left loop axioms
\[
\begin{cases}
(x-y)+y= x \\
 (x+y)-y=x
\end{cases}
\quad\text{or}\;\qquad
\begin{cases}
x+(-x+y)=y \\
 -x+(x+y)=y
\end{cases}
\]
in the imaginary context; we consider the left-hand side axioms. Table~\ref{Fig:RightLoops} gives the right loop axioms and their corresponding ``imaginary'' commutative diagrams.
\begin{table}
\begin{tabular}{c@{\qquad}c}
\toprule
right loop axiom & ``imaginary'' commutative diagram\\
\hline
$(x-y)+y=x$ & \iL1 $\xymatrix@C=50pt@R=30pt{Y\times Y \ar@{-->}[r]^-{\langle q,\overline{\pi_2}\rangle} \ar@{-->}[dr]_-{\overline{\pi_1}} & Y\times Y \ar@{-->}[d]^-{\mu^Y}\\
 & Y}$\\
\hline
$(x+y)-y=x$ & \iL2 $\xymatrix@C=50pt@R=30pt{Y\times Y \ar@{-->}[r]^-{\langle \mu^Y,\overline{\pi_2}\rangle} \ar@{-->}[dr]_-{\overline{\pi_1}} & Y\times Y \ar@{-->}[d]^-{q}\\
 & Y}$ \\
\bottomrule
\end{tabular}
\medskip
\caption{The right loop axioms and their corresponding diagrams}\label{Fig:RightLoops}
\end{table}
The only difference in the diagrams is that $\mu^Y$ and $q$ are swapped, just as ``$+$'' and ``$-$'' are swapped in the right loop axioms.

The commutativity of \iL1 follows from composing \iSs1 with $\pi_1$. From \eqref{naturality of mu} we know that $\pi_1\ci \mu^{Y\times Y}=\mu^Y\ci (\pi_1\times \pi_1)$. Then, we just have to prove that
\[
(\pi_1\times \pi_1)\ci \langle \langle 1_Y,0\rangle\ci q, \overline{\langle 1_Y,1_Y\rangle \pi_2}\rangle=\langle q,\overline{\pi_2}\rangle.
\]
In fact, $(\pi_1\times \pi_1)\ci \langle \langle 1_Y,0\rangle\ci q, \overline{\langle 1_Y,1_Y\rangle \pi_2}\rangle$ corresponds to the real morphism
$$
(\pi_1\times \pi_1)\langle \langle 1_Y,0\rangle q, \langle 1_Y,1_Y\rangle \pi_2\varepsilon_{Y\times Y}\rangle = \langle q, \pi_2\varepsilon_{Y\times Y}\rangle = \langle q,\overline{\pi_2}\rangle.
$$

The commutativity of \iL2 follows from \iSs2. In this case we must show that $\mu^{Y\times Y}\ci(\langle 1_Y,0\rangle\times \langle 1_Y,1_Y\rangle )=\langle \mu^Y,\overline{\pi_2}\rangle$. The imaginary morphism
\[
\mu^{Y\times Y}\ci(\langle 1_Y,0\rangle\times \langle 1_Y,1_Y\rangle )
\]
corresponds to the real morphism
\begin{align*}
 &\bi{1_{Y\times Y}}{1_{Y\times Y}}t_{Y\times Y,Y\times Y} P(\langle 1_Y,0\rangle \times \langle 1_Y,1_Y\rangle)\\
 &\stackrel{\mathclap{\eqref{naturality of t}}}{=}\quad
 \bi{\langle 1_Y,0\rangle}{\langle 1_Y,1_Y\rangle} t_{Y,Y}
 = \langle \bi{1_Y}{1_Y}, \bi{0}{1_Y}\rangle t_{Y,Y} \\
 &\stackrel{\mathclap{\eqref{imaginary addition},\eqref{pp2 for t}}}{=} \quad\langle \mu^Y,\pi_2\varepsilon_{Y\times Y}\rangle
 =\langle \mu^Y,\overline{\pi_2}\rangle.
\end{align*}

The converse is also true. Indeed, suppose the object $Y$ has the structure of an imaginary one-sided loop, in the sense that it is equipped with an imaginary morphism $q \colon Y \times Y \dashto Y$ which, together with the imaginary addition $\mu^Y$, satisfies \iL1 and \iL2. Then $q$ is the imaginary Schreier retraction for the split epimorphism \eqref{point S-special object}. To show this, we need to show that \iSs1 and \iSs2 hold. \iSs2 follows immediately from \iL2, because, as we already observed, $\mu^{Y\times Y}\ci(\langle 1_Y,0\rangle\times \langle 1_Y,1_Y\rangle )=\langle \mu^Y,\overline{\pi_2}\rangle$. In order to prove \iSs1, we use the previous equality $(\pi_1\times \pi_1)\ci \langle \langle 1_Y,0\rangle\ci q, \overline{\langle 1_Y,1_Y\rangle \pi_2}\rangle=\langle q,\overline{\pi_2}\rangle$ to get
\begin{align*}
\pi_1 \ci \mu^{Y \times Y} \ci \langle \langle 1_Y, 0 \rangle \ci q, \overline{\langle 1_Y, 1_Y \rangle \pi_2} \rangle
 & = \mu^Y \ci \langle q, \overline{\pi_2} \rangle \stackrel{{\iL1}}{=} \overline{\pi_1} = \pi_1\ci \overline{1_{Y\times Y}};
\end{align*}
also
\begin{align*}
& \pi_2 \ci \mu^{Y \times Y} \ci \langle \langle 1_Y, 0 \rangle\ci q, \overline{\langle 1_Y, 1_Y \rangle \pi_2} \rangle \\
 &\stackrel{\mathclap{\eqref{naturality of mu}}}{=} \;
 \mu^Y \ci (\pi_2 \times \pi_2) \ci \langle \langle 1_Y, 0 \rangle \ci q, \overline{\langle 1_Y, 1_Y \rangle \pi_2} \rangle\\
 &= \;\mu^Y \ci \langle \pi_2 \langle 1_Y, 0 \rangle\ci q, \overline{\pi_2 \langle 1_Y, 1_Y \rangle \pi_2} \rangle 
 = \mu^Y \ci \langle 0, \overline{\pi_2} \rangle \\
& = \; \mu^Y \ci \langle 0, 1_Y \rangle \ci \overline{\pi_2} 
 \;\stackrel{\mathclap{\eqref{0+x=x}}}{=} \;\overline{\pi_2} 
 \;= \; \pi_2\ci \overline{1_{Y\times Y}}.
\end{align*}
Combining these two equalities we get \iSs1. Hence

\begin{theorem}\label{one-sided loop}
In a regular unital category with binary coproducts, comonadic covers and a natural imaginary splitting, an object is an intrinsic Schreier special object if and only if its canonical imaginary magma structure is a one-sided loop structure.\noproof
\end{theorem}

\section{A non-varietal example}\label{A non-varietal example}
In this section we give an example of a non-varietal category for which there exists a natural imaginary splitting, and we analyse what are the intrinsic Schreier split epimorphisms and the intrinsic Schreier special objects in that context.

Take $\C=\PSetop$, which is a semi-abelian category~\cite{SAC,Bourntopos, Borceux-Bourn}, so it is a regular unital category with binary coproducts. We consider the power-set monad $(P,\delta,\varepsilon)$ in $\PSet$, where:
\begin{align*}
P(X,x_0)&=(P(X)=\{A\subseteq X\mid x_0\in A\}, \{x_0\}),\\
 \varepsilon_{(X,x_0)}(x)&=\{x,x_0\}, \\
 \delta_{(X,x_0)}( \{A_i\}_{i\in I} ) &= \bigcup_{i\in I} A_i, \;\text{where each $A_i\in P(X)$.}
\end{align*}

The monad $(P,\delta,\varepsilon)$ may be seen as a comonad in $\PSetop$. Moreover, it is easy to check that each $P(X,x_0)$ is projective in $\PSetop$, so that $\PSetop$ is equipped with comonadic projective covers; we are in the conditions of Subsection~\ref{Natural imaginary splittings}. A natural imaginary splitting in $\PSetop$ corresponds to a natural transformation $t$ in $\PSet$. We define, for any pair of pointed sets $(A,*)$ and $(B,*)$,
\begin{equation}\label{ex of t}
 t_{A,B}\colon (A\times B, (*,*)) \to (P(A+B),\{*\})\colon (a,b) \mapsto \{\underline{a},\overline{b},*\}
\end{equation}
It is easy to check that $t$ is a natural transformation and that it satisfies the opposite of equality \eqref{im splitting}, for all pointed sets $(A,*)$ and $(B,*)$, namely the condition that the diagram
$$\xymatrix@=30pt{& (P(A+B),\{*\}) \\
 (A\times B, (*,*)) \ar[ur]^-{t_{A,B}} & (A+B, *) \ar[l]^-{\matriz{1_A}{0}{0}{1_B}} \ar[u]_-{\varepsilon_{A+B}}}
$$
commutes in $\PSet$.

An intrinsic Schreier split epimorphism in $\PSetop$ corresponds to a split monomorphism in $\PSet$. The following diagram represents a split monomorphism, given by an injection $f$, and its cokernel in $\PSet$
$$
\xymatrix@!0@=6em{(K,*) & (X,*) \ar[l]^-{k} \ar@<.5ex>@{>>}[r]^-{s} & (Y,*), \ar@{{ >}->}@<.5ex>[l]^-{f}}
$$
where $K=X\backslash Y\cup\{*\}$, $k(y)=*$, for all $y\in Y$ and $k(x)=x$, for all $x\in X\backslash Y$. It is an intrinsic Schreier split monomorphism if there exists a morphism of pointed sets $q\colon (K,*) \to (P(X),\{*\})$ such that the opposite of equalities \iS1 and \iS2 hold. Note that $q(x)\in P(X)$, i.e., $*\in q(x)\subseteq X$, for all $x\in X\backslash Y$.

The opposite of \iS1 is given by the commutativity of the following diagram (we omit the fixed points to make it easier to read)
$$
\xymatrix@C=45pt@R=30pt{P^2(X) \ar[d]_-{\delta_{X}} & \ar[l]_-{P(\bi{1}{1})} P(P(X)+P(X)) & \ar[l]_-{t_{P(X),P(X)}} P(X)\times P(X) \\
 P(X) & & \ar[ll]^-{\varepsilon_{X}} X. \ar[u]_-{\langle qk, \varepsilon_X fs \rangle}}
$$
The commutativity of the diagram above always holds for any element $y\in Y$. For any element $x\in X\backslash Y$, we get
$$
\resizebox{\textwidth}{!}{
\xymatrix@C=40pt@R=30pt{ \{ q(x), \{fs(x),*\}, \{*\}\} \ar@<-4ex>[d]_-{\delta_{X}} & \{\underline{q(x)}, \overline{\{fs(x),*\}},\{*\}\} \ar[l]_-{P(\bi{1}{1})}
 & (q(x), \{fs(x),*\}) \ar[l]_-{t_{P(X),P(X)}} \\
 q(x)\cup\{fs(x),*\} = \{x,*\} & & x. \ar[ll]^-{\varepsilon_{X}} \ar[u]_-{\langle qk, \varepsilon_X fs \rangle}}}
$$
From the equality $q(x)\cup\{fs(x),*\} = \{x,*\}$, and the fact that $s(x)\in Y$ and $x\in X \backslash Y$, we deduce that $s(x)=*$ and $q(x)=\{x,*\}$. As a consequence the split monomorphism is isomorphic to the binary coproduct
$$
\xymatrix@R=30pt@C=30pt{ (X\backslash Y\cup\{*\},*) \ar@{=}[d] & (X,*) \ar@<.5ex>[r]^-{s} \ar@<.5ex>[l]^-{k} \ar[d]|-{\cong}
 & (Y,*) \ar@<.5ex>@{(^->}[l]^-f \ar@{=}[d] \\
 (X\backslash Y\cup\{*\},*) & ((X\backslash Y\cup\{*\}) + Y ,*)\ar@<.5ex>[r]^-{\bi{0}{1_Y}} \ar@<.5ex>[l]^-{\bi{1_K}{0}}
 & (Y,*) \ar@<.5ex>@{(^->}[l]^-{\iota_2}}
$$
It is easy to see that the opposite of equality \iS2 always holds.

We have just proved that in $\PSetop$, with respect to the natural imaginary splitting \eqref{ex of t}, the only intrinsic Schreier split epimorphisms correspond to binary product projections. Moreover, a pointed set $(Y,*)$ is an intrinsic Schreier special object if and only if \eqref{point S-special object} is a product projection, i.e., if and only if it is the zero object.

Note that we could also apply the same approach to the finite power-set monad in $\PSet$.

\section{Intrinsic Schreier special objects vs.\ protomodular objects}\label{isso vs proto objs}
Recall from~\cite{2Chs} that an object $Y$ in a finitely complete category is called a \defn{protomodular object} when all points over it $(f\colon X\to Y, s\colon Y\to X)$ are stably strong. More precisely, for any pullback
$$\xymatrix@=40pt{Z\times_Y X \ar[r] \ar@<-2pt>[d]_(.55){\pi_Z} \ophalfsplitpullback & X \ar@<-2pt>[d]_-f \\
 Z \ar[r]_-g \ar@<-2pt>[u]_(.45){\langle 1_Z,sg\rangle} & Y, \ar@<-2pt>[u]_-s}
$$
the pair $(\langle 0,k\rangle, \langle 1_Z, sg\rangle)$, where $k$ is the kernel of $f$, is a jointly ex\-tremal-epimorphic pair. If the point $(f,s)$ is stably strong, then it is strong, i.e., $(k,s)$ is a jointly extremal-epimorphic pair.

In the category $\Mon$ of monoids the notion of a Schreier special object and the notion of a protomodular object both coincide with that of a group: a monoid is a Schreier special object if and only if it is a group~\cite{S-proto} if and only if it is a protomodular object~\cite{2Chs}.

The question of understanding under which conditions these two notions coincide arises naturally. In general, neither of these notions implies the other, as we showed in~\cite{ise}. Indeed, the variety $\HSLat$ of Heyting semi\-lattices provides an example of a category where all objects are protomodular, but not every object is Schreier special (\cite{ise}, Example 7.7).

On the other hand, the cyclic group $C_2=(\{0,1\},+)$ gives an example of a Schreier special object in the J\'onsson--Tarski variety of unitary magmas $\Mag$, because it is a right loop. However, we gave an example of a point $X \leftrightarrows C_2$ which is not strong. Consequently, $C_2$ is not a protomodular object (\cite{ise}, Example 7.4). Of key importance here is that the unitary magma $(X,+)$ is \emph{non-associative}. As we shall prove next in Corollary~\ref{iSchreier special => proto}, the presence of the associativity axiom (Section~\ref{The associativity axiom}) gives a link between intrinsic Schreier special objects and protomodular objects: then
\begin{quote}
\emph{Every intrinsic Schreier special object is a protomodular object.}
\end{quote}

The proof of this statement follows the same proof for monoids, i.e., that a Schreier special monoid $Y$ is necessarily a group; the inverse of an element $y\in Y$ is given by $q(0,y)$, where $q$ is the imaginary retraction for \eqref{point S-special object} (see Proposition 3.1.6 of~\cite{SchreierBook}). Also, all points over a group are necessarily Schreier split epimorphisms (see Corollary 3.1.7 of~\cite{SchreierBook}).

\begin{lemma}
If \eqref{point S-special object} satisfies \iSs1, then $\mu^Y\diam \langle q\ci \langle 0,1_Y\rangle, \overline{1_Y}\rangle = \overline{0}$:
\begin{equation}\label{q(0,y)+y=0}
\vcenter{\xymatrix@C=60pt@R=30pt{Y \ar@{-->}[r]^-{\langle q\ci \langle 0,1_Y\rangle, \overline{1_Y}\rangle} \ar@{-->}[dr]_-{\overline{0}} & Y\times Y. \ar@{-->}[d]^-{\mu^Y}\\
 & Y}}
\end{equation}
\end{lemma}
\begin{proof} In Subsection~\ref{Imaginary (one-sided) loops} we saw that \iL1 follows from \iSs1. If we precompose \iL1 with $\langle 0,1_Y\rangle\colon Y\to Y\times Y$, we get
\[
\mu^Y\diam \langle q, \overline{\pi_2}\rangle \ci \langle 0,1_Y\rangle = \overline{\pi_1}\ci \langle 0,1_Y\rangle 
 \quad\Leftrightarrow\quad \mu^Y\diam \langle q\ci \langle 0,1_Y\rangle, \overline{1_Y}\rangle = \overline{0}\text{.} \qedhere
\]
\end{proof}

\begin{lemma}
Suppose that the natural addition $(\mu^X\colon X\times X\dashto X)_{X\in \C}$ is associative. If $Y$ is an intrinsic Schreier special object, then
\begin{equation}\label{y+q(0,y)=0}
\vcenter{\xymatrix@C=60pt@R=30pt{Y \ar@{-->}[r]^-{\langle \overline{1_Y}, q\ci \langle 0,1_Y\rangle\rangle} \ar@{-->}[dr]_-{\overline{0}} & Y\times Y \ar@{-->}[d]^-{\mu^Y}\\
 & Y}}
\end{equation}
$\mu^Y\diam \langle \overline{1_Y}, q\ci \langle 0,1_Y\rangle\rangle = \overline{0}$.
\begin{proof} In Subsection~\ref{Imaginary (one-sided) loops} we saw that \iL2 follows from \iSs2. If we precompose \iL2 with $\langle \mu^Y\diam \langle \overline{1_Y}, q\ci \langle 0,1_Y\rangle\rangle,\overline{1_Y}\rangle\colon Y\dashto Y\times Y$, we obtain
\begin{multline*}
q\diam \langle \mu^Y, \overline{\pi_2}\rangle\diam \langle \mu^Y\diam \langle \overline{1_Y}, q\ci \langle 0,1_Y\rangle\rangle,\overline{1_Y}\rangle
 = \overline{\pi_1}\diam \langle \mu^Y\diam \langle \overline{1_Y}, q\ci \langle 0,1_Y\rangle\rangle,\overline{1_Y}\rangle,
\end{multline*}
which is equivalent to
\begin{align*}
& q\diam \langle \mu^Y\diam \langle \mu^Y\diam \langle \overline{1_Y}, q\ci \langle 0,1_Y\rangle\rangle,\overline{1_Y}\rangle, \overline{1_Y}\rangle = \mu^Y\diam \langle \overline{1_Y}, q\ci \langle 0,1_Y\rangle\rangle\\
 &\stackrel{\mathclap{\eqref{assoc}}}{\Leftrightarrow} \quad q\diam \langle \mu^Y\diam \langle \overline{1_Y},\mu^Y\diam \langle q\ci \langle 0,1_Y\rangle, \overline{1_Y}\rangle \rangle, \overline{1_Y}\rangle = \mu^Y\diam \langle \overline{1_Y}, q\ci \langle 0,1_Y\rangle\rangle\\
 &\stackrel{\mathclap{\eqref{q(0,y)+y=0}}}{\Leftrightarrow} \quad q\diam \langle \mu^Y\diam \langle \overline{1_Y},\overline{0} \rangle, \overline{1_Y}\rangle = \mu^Y\diam \langle \overline{1_Y}, q\ci \langle 0,1_Y\rangle\rangle\\
 &\Leftrightarrow \quad q\diam \langle \mu^Y\ci \langle 1_Y,0\rangle, \overline{1_Y}\rangle = \mu^Y\diam \langle \overline{1_Y}, q\ci \langle 0,1_Y\rangle\rangle\\
 &\stackrel{\mathclap{\eqref{x+0=x}}}{\Leftrightarrow} \quad q\diam \langle \overline{1_Y}, \overline{1_Y}\rangle = \mu^Y\diam \langle \overline{1_Y}, q\ci \langle 0,1_Y\rangle\rangle\\
 &\Leftrightarrow \quad q\ci \langle 1_Y,1_Y\rangle = \mu^Y\diam \langle \overline{1_Y}, q\ci \langle 0,1_Y\rangle\rangle\\
 &\stackrel{\mathclap{\iSs4}}{\Leftrightarrow} \quad \overline{0} = \mu^Y\diam \langle \overline{1_Y}, q\ci \langle 0,1_Y\rangle\rangle.\qedhere
\end{align*}
\end{proof}
\end{lemma}

\begin{proposition}\label{Y special => iS1}
Suppose that the natural addition $(\mu^X\colon X\times X\dashto X)_{X\in \C}$ is associative. If $Y$ is an intrinsic Schreier special object, then any split epimorphism \eqref{iSchreier} satisfies \iS1.
\end{proposition}
\begin{proof} We define an imaginary morphism $\rho\colon X\dashto K$ through the universal property of the kernel:
$$
\xymatrix@C=40pt@R=30pt{K \ar@{ |>->}[r]^-k & X \ar@<-.5ex>@{>>}[r]_-{f} & Y. \ar@<-.5ex>[l]_-{s} \\
 & X \ar@{-->}[u]_(.4){\mu^X\ci (1_X\times s) \ci (1_X\times(q\ci \langle 0,1_Y\rangle))\ci \langle 1_X,f\rangle} \ar@{.>}[ul]^-{\rho}}
$$
Indeed,
\begin{align*}
& f\ci \mu^X\ci (1_X\times s) \ci (1_X\times(q\ci \langle 0,1_Y\rangle))\ci \langle 1_X,f\rangle\\
 &\stackrel{\mathclap{\eqref{naturality of mu}}}{=} \quad \mu^Y \ci (f\times f) \ci \langle \overline{1_X}, s\ci q\ci \langle 0,1_Y\rangle f \rangle \\
 &= \quad \mu^Y \diam \langle \overline{f}, q\ci \langle 0,1_Y\rangle f\rangle 
\; = \; \mu^Y \diam \langle \overline{1_Y}, q\ci \langle 0,1_Y\rangle \rangle \ci f
 \;\stackrel{\mathclap{\eqref{y+q(0,y)=0}}}{=} \; \overline{0}
\end{align*}

Now we must check \iS1 for \eqref{iSchreier}:
\begin{align*}
\mu^X\diam \langle k\ci \rho, \overline{sf}\rangle \; & = \;
 \mu^X\diam \langle \mu^X\diam \langle \overline{1_X}, s\ci q\ci \langle 0,1_Y\rangle f \rangle, \overline{sf}\rangle \\
 & \stackrel{\mathclap{\eqref{assoc}}}{=} \; \mu^X \diam \langle \overline{1_X}, \mu^X\diam \langle s\ci q\ci \langle 0,1_Y\rangle f,\overline{sf}\rangle \rangle \\
 & = \; \mu^X \diam \langle \overline{1_X}, \mu^X\ci (s\times s) \ci \langle q\ci \langle 0,1_Y\rangle,\overline{1_Y}\rangle\ci f \rangle \\
 & \stackrel{\mathclap{\eqref{naturality of mu}}}{=} \; \mu^X \diam \langle \overline{1_X}, s\ci \mu^Y \diam \langle q\ci \langle 0,1_Y\rangle,\overline{1_Y}\rangle\ci f \rangle \\
 & \stackrel{\mathclap{\eqref{q(0,y)+y=0}}}{=} \; \mu^X \diam \langle \overline{1_X}, \overline{0} \rangle
 \;= \; \mu^X\ci \langle 1_X,0\rangle
\; \stackrel{\mathclap{\eqref{x+0=x}}}{=} \; \overline{1_X}.\qedhere
\end{align*}
\end{proof}

\begin{corollary}\label{iSchreier special => proto} Suppose that the natural addition $(\mu^X\colon X\times X\dashto X)_{X\in \C}$ is associative. Then every intrinsic Schreier special object is a protomodular object.
\end{corollary}
\begin{proof} This follows from Proposition~\ref{Y special => iS1} and Proposition 5.8 in~\cite{ise}, which states that any split epimorphism satisfying \iS1 is stably strong.
\end{proof}

\begin{remark}\label{last remarks1}
Even if the natural addition $(\mu^X\colon X\times X\dashto X)_{X\in \C}$ is associative, the converse of Corollary~\ref{iSchreier special => proto} may be false. As mentioned above, the variety $\HSLat$ of Heyting semi\-lattices provides an example of a category where all objects are protomodular, but not all are Schreier special objects. The natural addition for $\HSLat$, given by the meet, is associative.
\end{remark}

\begin{remark}\label{last remarks2}
The variety $\Loop$ of (left and right) loops gives an example where the natural addition is non-associative. All loops are intrinsic Schreier special objects (see Section~\ref{Intrinsic Schreier special objects}) and they are also all protomodular objects (because $\Loop$ is a semi-abelian category, thus a protomodular category). So, the fact that all intrinsic Schreier special objects are protomodular objects does not imply that the natural addition is associative.
\end{remark}

From the remark above, in $\Gp$ all objects are intrinsic Schreier special with respect to its usual group operation. Also, all objects are protomodular since $\Gp$ is a protomodular category. So the two notions coincide in $\Gp$, just as in the case of~$\Mon$. However, in $\Mon$ there are only two possible choices of imaginary splittings (see Subsection~\ref{Natural imaginary splittings}). In $\Gp$ there are countably many possibilities. Given groups $X$ and $Y$, a natural imaginary splitting $t\colon P(X\times Y)\to X+Y$ may be defined by making $t([(x,y)])$ equal to any combination of alternating products of $\underline{x}$ or $\underline{x}^{-1}$, and of $\overline{y}$ or $\overline{y}^{\,-1}$, for which the products of the $x$'s gives $x$ and the products of the $y$'s gives $y$. For example $\underline{x}^{-1}\overline{y}\underline{x}^2$ or $\underline{x}\overline{y}\underline{x}^{-1}\overline{y}^{\,-1}\underline{x}\overline{y}$.

Although these notions are independent in general, as we have already observed, there are special properties of the category of groups that make the notions coincide. From Corollary~\ref{iSchreier special => proto}, we know that the associativity of the group operation implies that all intrinsic Schreier special objects are protomodular objects. This associativity is not enough to guarantee that every protomodular object is intrinsic Schreier special (Remark~\ref{last remarks2}). This leads us to the following question:
\begin{quote}
\emph{What property of $\Gp$ guarantees that all protomodular objects are intrinsic Schreier special ones?}
\end{quote}
We cannot answer this question now, but we can see that groups lack a certain homogeneity, in the sense that the concept of an intrinsic Schreier special object strongly depends on the chosen natural imaginary splitting. We can eliminate this discrepancy by considering groups which satisfy the property with respect to \emph{all} natural imaginary splittings. Then we find:

\section{The variety of $2$-Engel groups}

The aim of this section is to show that $2$-Engel groups are intrinsic Schreier special objects with respect to \emph{all} natural imaginary splittings: Proposition~\ref{iSchreier for all t}.

We begin by recalling the definition and main properties of $2$-Engel groups needed in the sequel, which can be found in~\cite{Burnside, Kappe WP, Kappe LC}.

Here we denote the conjugation of an element $x$ by an element $y$ as ${}^yx=y x y^{-1}$ and we write $[x,y]$ for $xyx^{-1}y^{-1}={}^xyy^{-1}$. Then
\begin{align*}
[xy,z] &= xyzy^{-1}x^{-1}z^{-1} = xyzy^{-1}z^{-1}x^{-1}xzx^{-1}z^{-1} \\&= x[y,z]x^{-1}[x,z] = {}^x[y,z][x,z].
\end{align*}
Likewise, $[x,yz]=[x,y]{}^y[x,z]$. Note also that $[x,y]^{-1}=[y,x]$.

\begin{proposition} \label{equivalent properties to 2-Engel}
For a group $E$, the following conditions are equivalent:
\begin{tfae}
\item $[[x,y],y]=1$ for all $x$, $y\in E$;
\item $[[x,y],x]=1$ for all $x$, $y\in E$;
\item $[{}^yx,x]=1$ for all $x$, $y\in E$.
\end{tfae}
\end{proposition}
\begin{proof}
First note that $[[x,y],y]=1$ for all $x$, $y\in E$ if and only if $[[y,x],x]=1$ for all $x$, $y\in E$. Now
\[
1=[1,x]=[[x,y][x,y]^{-1},x]=[[x,y][y,x],x]={}^{[x,y]}[[y,x],x][[x,y],x]
\]
so that $[[y,x],x]=1$ if and only if $[[x,y],x]=1$, and i.\ and ii.\ are equivalent. Finally, 
\[
[[y,x],x] = [{}^yxx^{-1},x] = {}^{{}^yx}[x^{-1}, x][{}^yx,x] = [{}^yx, x],
\]
which shows that $[[y,x],x]=1$ if and only if $[{}^yx,x]=1$, and i.\ and iii.\ are equivalent.
\end{proof}

\begin{definition}\label{2-Engel group}\label{2-Engel}A group $E$ is called a \defn{$2$-Engel group} if it satisfies any of the equivalent conditions of the previous proposition.
\end{definition}

\begin{example}
\begin{enumerate}
 \item Any abelian group is obviously a $2$-Engel group.
 \item The group of quaternions ${Q_8}$ is a $2$-Engel group which is not abelian.
 \item The smallest non $2$-Engel (thus non-abelian) group is the symmetric group ${S_3}$ (which is isomorphic to the dihedral group ${D_6}$).
 \item The dihedral group ${D_8}$ is $2$-Engel, but the dihedral group ${D_{10}}$ is not (see Example~\ref{counterexample}).
\end{enumerate}
\end{example}

\begin{lemma}\label{lemma for 2-Engel} Let $E$ be a $2$-Engel group. Then:
\begin{enumerate}
\item $[xy,z]={}^{x}([y,z][x,z])$, for all $x$, $y$, $z\in E$;
\item $[x,yz]={}^{y}([x,y][x,z])$, for all $x$, $y$, $z\in E$.
\end{enumerate}
\end{lemma}
\begin{proof}
1. $[xy,z]={}^{x}[y,z][x,z]=x[y,z]x^{-1}[x,z]\stackrel{\text{\ref{equivalent properties to 2-Engel}.ii}}{=} x[y,z][x,z]x^{-1}$.

2. $[x,yz]=[x,y]{}^{y}[x,z]=[x,y]y[x,z]y^{-1}\stackrel{\text{\ref{equivalent properties to 2-Engel}.i}}{=} y[x,y][x,z]y^{-1}$.
\end{proof}

\begin{proposition} \label{more equivalences to 2-Engel}
For a group $G$, the following conditions are equivalent:
\begin{tfae}
\item $G$ is a $2$-Engel group;
\item $[x,y^{-1}]=[x,y]^{-1}$ for all $x$, $y \in G$;
\item $[x^{-1},y]=[x,y]^{-1}$ for all $x$, $y \in G$;
\item $[x,y^k]=[x,y]^k$, for all $x$, $y \in G$ and all $k\in \Z$;
\item $[x^k,y]=[x,y]^k$, for all $x$, $y \in G$ and all $k\in \Z$.
\end{tfae}

\end{proposition}
\begin{proof}
It is clear that ii.\ and iii.\ are equivalent, as are iv.\ and v.. It is also clear that ii.\ is a special case of iv.. If now i.\ holds, then by Lemma~\ref{lemma for 2-Engel} we may deduce
\[
1=[1,y]=[xx^{-1},y]={}^x([x^{-1},y][x,y])
\]
from which it follows that $[x^{-1},y][x,y]=1$, and iii.\ holds. On the other hand, ii.\ implies i.\ via
\[
[[x,y],y]=[x,y]y[x,y^{-1}]y^{-1}=xyx^{-1}y^{-1}yxy^{-1}x^{-1}yy^{-1}=1. 
\]
For positive $k$, v.\ follows from i.\ by induction. The cases $k=0$ and $k=1$ being clear, we assume that $[x^k,y]=[x,y]^k$ for some $k\geq 1$. Then Lemma~\ref{lemma for 2-Engel} and Proposition~\ref{equivalent properties to 2-Engel} together imply
\[
[x^{k+1},y]=[xx^k,y]={}^x([x^{k},y][x,y])={}^x([x,y]^{k}[x,y])=
{}^x([x,y]^{k+1})=[x,y]^{k+1}.
\]
If $k$ is negative then we first apply iii..
\end{proof}

It is known (and in fact not hard to see, using the above results) that the free object on two generators in the variety $\Eng_2(\Gp)$ of $2$-Engel groups is $2$-nilpotent. This allows us to prove the following result.

\begin{lemma}\label{Binary commutators distribute}
In a $2$-Engel group $E$, consider $x$, $y\in E$ and the subgroup $H$ generated by $\{x,y\}$. Let $a, b, c \in H$. Then:
\begin{enumerate}
\item $[[a,b],c]=1$;
\item $[ab,c]=[b,c][a,c]$ and $[a,bc]=[a,b][a,c]$;
\item $[a^{-1},b]=[a,b]^{-1}=[a,b^{-1}]$;
\item $[a^k,b]=[a,b]^k=[a,b^k]$, for all $k\in \Z$.
\end{enumerate}
\end{lemma}
\begin{proof}
1.\ Follows from the fact that the free $2$-Engel group on two generators $x$ and~$y$ is necessarily $2$-nilpotent. 2.\ $[ab,c]={}^a[b,c][a,c]=a[b,c]a^{-1}[a,c]\stackrel{1.}{=}[b,c][a,c]$. The proof of the second claim is similar. 3.\ and 4.\ follow from Proposition~\ref{more equivalences to 2-Engel}.
\end{proof}

We now look at a specific natural imaginary splitting in $\Gp$: the one defined by the function
\begin{equation}\label{simple t}
t_{X,Y}\colon X\times Y\ito X+Y\colon (x,y)\mapsto \underline{x}^{-1}\overline{y}\underline{x}^2,
\end{equation}
for any pair of groups $X$ and $Y$. It is easy to see that this $t$ is indeed a natural imaginary splitting. When $X=Y$, we write
$$x* y= \mu^Y(x,y)=\nabla_Y(t_{Y,Y}(x,y))=x^{-1}yx^2.$$
It is easy to check that $x*1=x=1*x$ and that $x*x^{-1}=1=x^{-1}*x$; however $*$ is not associative.

A group $Y$ is an intrinsic Schreier special object with respect to \eqref{simple t} if there exists an imaginary retraction $q\colon Y\times Y\ito Y$ such that \iSs1 and \iSs2 hold. In this case
\begin{enumerate}
\item[\iSs1] means that $q(x,y)*y=x$, for all $x$, $y\in Y$, and
\item[\iSs2] means that $q(x*y,y)=x$, for all $x$, $y\in Y$
\end{enumerate}
---see Section~\ref{Intrinsic Schreier special objects}.

\begin{proposition}\label{2-Engel is iSspecial wrt simple t}
If $Y$ is a $2$-Engel group, then $Y$ is an intrinsic Schreier special object with respect to the natural imaginary splitting \eqref{simple t}.
\end{proposition}
\begin{proof}
We define the imaginary retraction by $q(x,y)=x*y^{-1}$. Then, for all $x$, $y\in Y$, \iSs1 holds:
\begin{align*}
q(x,y)*y & = (x*y^{-1})*y \\
 & = (x^{-1} y^{-1} x^2) * y 
 = x^{-2}yxyx^{-1}y^{-1}x^2x^{-1}y^{-1}x^2 \\
 & = x^{-2}y [x,y] xy^{-1}x^2 
 \;\stackrel{\text{\ref{more equivalences to 2-Engel}.iii}}{=} \;x^{-2}y [y^{-1},x] xy^{-1}x^2 \\
 & = x^{-2} y y^{-1}xyx^{-1}xy^{-1}x^2 = x.
\end{align*}
As for \iSs2, the equality $q(x*y,y)=(x*y)*y^{-1} = x$ holds by swapping $y$ and $y^{-1}$ in \iSs1.
\end{proof}

\begin{example}\label{counterexample} The dihedral group ${D_{10}}$ is generated by elements $a$ and $b$ such that $a^5=1$, $b^2=1$ and $abab=1$. We have $$ {D_{10}}=\{1,a,a^2,a^3,a^4, b, ab, a^2b, a^3b, a^4b\},$$
where the elements $b$, \dots, $a^4b$ are all inverses to themselves. We have
\begin{itemize}
 \item ${D_{10}}$ is not a $2$-Engel group: $[a,ab]ab = a^2$, while $ab[a,ab]=a^4b$.
 \item ${D_{10}}$ is an intrinsic Schreier special object with respect to the natural imaginary splitting~\eqref{simple t}. It suffices to build the Cayley table for the product $*$ and observe that it gives a Latin square. The fact that it is a Latin square guarantees the existence of a unique element, which is equal to $q(x,y)$, satisfying the equality \iSs1 $q(x,y)*y=x$. The equality \iSs2 follows from the uniqueness of each $q(x,y)$.
 \item ${D_{10}}$ is not an intrinsic Schreier special object with respect to the natural imaginary splitting which gives rise to $x*'y=[x,y]^2xy$. For example, $q(1,b)$ should be the unique element of ${D_{10}}$ such that $q(1,b)*' b=1$. However, all of the elements $b$, \dots, $a^4b$ satisfy this equality.
\end{itemize}
\end{example}

This example shows that the converse of Proposition~\ref{2-Engel is iSspecial wrt simple t} is false. However, we may claim the following:

\begin{proposition}\label{Y iSspecial wrt simple t is 2-Engel}
If a group $Y$ is an intrinsic Schreier special object with respect to the natural imaginary splitting
\eqref{simple t} and such that $q(x,y)=x*y^{-1}$, then $Y$ is a $2$-Engel group.
\end{proposition}
\begin{proof}
It suffices we use \iSs1
\[
(x*y^{-1})*y=x
\]
to see that $Y$ is in $\Eng_2(\Gp)$. Indeed, this is equivalent to
\begin{align*}
x=(x^{-1}y^{-1}x^{2})*y&=x^{-2}yx y x^{-1}y^{-1}x^{2}x^{-1}y^{-1}x^{2}\\
&=x^{-2}yxyx^{-1}y^{-1}xy^{-1}x^{2},
\end{align*}
which may be rewritten as $1=x^{-2}yxyx^{-1}y^{-1}xy^{-1}x$, so
$1=x^{-1}yxyx^{-1}y^{-1}xy^{-1}$. This gives
\[
y^{-1}xyx^{-1}=xyx^{-1}y^{-1},
\]
or, equivalently, $[y^{-1},x]=[x,y]=[y,x]^{-1}$. The result now follows from Proposition~\ref{more equivalences to 2-Engel}.
\end{proof}

Next we aim to prove the that a $2$-Engel group~$Y$ is an intrinsic Schreier special object with respect to \emph{all} natural imaginary splittings $t$. So, we need to extend Proposition~\ref{2-Engel is iSspecial wrt simple t} to all $t$.

\begin{lemma}\label{Naturality gives easy splitting}
If $t$ is a natural imaginary splitting in $\Gp$, then for each pair of groups $X$, $Y$ and all $x\in X$, $y\in Y$ we have that $t_{X,Y}(x,y)\in X+Y$ may be written as a product
\[
\underline{x}^{k_1}\overline{y}^{l_1}\cdots \underline{x}^{k_m}\overline{y}^{l_m},
\]
for some $m\in \mathbb{N}$ and $k_1$, \dots, $k_m$, $l_1$, \dots, $l_m\in \Z$ such that
\[
\sum_{1\leq i\leq m}k_i=1=\sum_{1\leq i\leq m}l_i.
\]
\end{lemma}
\begin{proof}
If $X=Y=\Z$, then $t_{\Z,\Z}(k,l)\in \Z+\Z$ must be of the form $\underline{k_1}\overline{l_1}\cdots \underline{k_m}\overline{l_m}$, for some $m\in \mathbb{N}$ and $k_1$, \dots, $k_m$, $l_1$, \dots, $l_m\in \Z$ such that $\sum_{1\leq i\leq m}k_i=k$ and $\sum_{1\leq i\leq m}l_i=l$, for all $(k,l)\in \Z\times \Z$ (see Subsection~\ref{Natural imaginary splittings}). Consider the group homomorphisms $f\colon \Z \to X \colon 1 \mapsto x$ and $g\colon \Z \to Y \colon 1 \mapsto y$. The naturality of $t$ gives the commutative diagram (see \eqref{naturality of t})
$$\xymatrix{\Z\times \Z \ar@{-->}[r]^-{t_{\Z,\Z}} \ar[d]_-{f\times g} & \Z+\Z \ar[d]^-{f+g} \\
 X\times Y \ar@{-->}[r]_-{t_{X,Y}} & X+Y,}
$$
from which we conclude that $t_{X,Y}(x,y)=t_{X,Y}(f\times g)(1,1)=(f+g)t_{\Z,\Z}(1,1)$. Suppose that $t_{\Z,\Z}(1,1)= \underline{k_1}\overline{l_1}\cdots \underline{k_m}\overline{l_m}$, where $\sum_{1\leq i\leq m}k_i=1=\sum_{1\leq i\leq m}l_i$. We get
$t_{X,Y}(x,y)=\underline{x}^{k_1}\overline{y}^{l_1}\cdots \underline{x}^{k_m}\overline{y}^{l_m}$, as desired.
\end{proof}

\begin{proposition}\label{Proposition Simple *}
If $Y$ is a $2$-Engel group and $t$ is a natural imaginary splitting in $\Gp$, then the induced operation $x* y= \mu^Y(x,y)=\nabla_Y(t_{Y,Y}(x,y))$ may be written as
\[
x*y=[x,y]^kxy
\]
for some $k\in \Z$.
\end{proposition}
\begin{proof}
Lemma~\ref{Naturality gives easy splitting} tells us that
\[
x*y={x}^{k_1}{y}^{l_1}\cdots {x}^{k_m}{y}^{l_m},
\]
for some $m\in \mathbb{N}$ and $k_1$, \dots, $k_m$, $l_1$, \dots, $l_m\in \Z$ such that $\sum_{1\leq i\leq m}k_i=1=\sum_{1\leq i\leq m}l_i$. We rewrite the expression above as
\[
x*y=({x}^{k_1}{y}^{l_1}\cdots {x}^{k_m}{y}^{(l_m-1)}x^{-1}y^{0})xy,
\]
 where the product in brackets is such that the sums of the exponents of the $x$'s and $y$'s are zero. Hence this expression is a commutator word in~$x$ and~$y$: it is a product of (nested) commutators. By Lemma~\ref{Binary commutators distribute}, all higher-order commutators in this product vanish; furthermore, the expression is equal to a product of commutators of the form $[x,y]$ or $[y,x]=[x,y]^{-1}$. Hence it is of the form $[x,y]^k$ for some integer~$k$.
\end{proof}

\begin{proposition}\label{iSchreier for all t}
If $Y$ is a $2$-Engel group, then $Y$ is intrinsic Schreier special with respect to all natural imaginary splittings in $\Gp$.
\end{proposition}
\begin{proof}
The proof is similar to that of Proposition~\ref{2-Engel is iSspecial wrt simple t}. We define the imaginary retraction by $q(x,y)=x*y^{-1}=[x,y^{-1}]^kxy^{-1}$ (Proposition~\ref{Proposition Simple *}). We use Propositions~\ref{equivalent properties to 2-Engel}, \ref{more equivalences to 2-Engel} and Lemma~\ref{Binary commutators distribute} to prove that \iSs1 holds:
\begin{align*}
(x*y^{-1})*y &= ([x,y^{-1}]^k\, xy^{-1})*y\\
&=\bigl[[x,y^{-1}]^kxy^{-1},y\bigr]^k\, [x,y^{-1}]^k\, xy^{-1}y\\
&=\bigl[[x, y]^{-k}xy^{-1}, y\bigr]^k [x,y]^{-k}x\\
&\stackrel{\mathclap{\ref{Binary commutators distribute}}}{=}\bigl([y^{-1},y][x,y][[x, y]^{-k}, y]\bigr)^k [x,y]^{-k}x\\
&=\bigl([x,y][[x,y], y]^{-k}\bigr)^k [x,y]^{-k}x\\
&= [x,y]^k [x,y]^{-k}x
=x.
\end{align*}
As for \iSs2, the equality $q(x*y,y)=(x*y)*y^{-1}= x$ follows from \iSs1 by replacing $y$ with $y^{-1}$.
\end{proof}

It remains an open question whether or not the converse of Proposition~\ref{iSchreier for all t} holds; we are currently working on this question. Essentially the same result holds for Lie algebras, as we shall explain now.

\section{Lie algebras}
Let $\K$ be a field, and consider the variety $\Lie_{\K}$ of $\K$-Lie algebras. Recall that a \defn{$2$-Engel} Lie algebra is a Lie algebra $\e$ that satisfies the commutator identity $[[x,y],y]=1$ for all $x$, $y\in \e$. The aim of this section is to relate the variety $\Eng_2(\Lie_\K)$ of $2$-Engel Lie algebras over $\K$ to the Schreier special objects with respect to all natural imaginary splittings: Theorem~\ref{Thm:2EngelLie}. We can actually just follow the pattern of the previous section; since furthermore things are somewhat simpler here, we will only sketch the basic idea.

We may proceed as in Proposition~\ref{2-Engel is iSspecial wrt simple t}, now taking the natural imaginary splitting in $\Lie_{\K}$ defined by
\begin{equation*}\label{simple t Lie}
t_{\x,\y}\colon \x\times \y\ito \x+\y\colon (x,y)\mapsto \underline{x}+\overline{y}+[\underline{x},\overline{y}].
\end{equation*}
Recall that the free Lie algebra over $\K$ on a single generator is $\K$ itself, equipped with the trivial bracket. Mimicking the proof of Lemma~\ref{Naturality gives easy splitting}, we see that for any pair of $\K$-Lie algebras $\x$ and $\y $ and any $x\in \x$, $y\in \y$, necessarily
\[
t_{\x,\y}(x,y) = \underline{x} + \overline{y} + \phi(\underline{x},\overline{y}),
\]
where $\phi(\underline{x},\overline{y})$ is an expression in terms of Lie brackets of $\underline{x}$'s and $\overline{y}$'s. Now using essentially the same proof as in groups, we see that if~$\y$ is $2$-Engel, all higher-order brackets vanish, and we deduce that
\[
t_{\x,\y}(x,y) = \underline{x} + \overline{y} + k[\underline{x},\overline{y}]
\]
for some $k \in \K$. As in Proposition~\ref{Proposition Simple *}, it follows that $x*y = x + y + k[x,y]$. It is then again easy to check that \iSs1 and \iSs2 hold.

\begin{proposition}\label{Thm:2EngelLie}
Any $2$-Engel $\K$-Lie algebra is intrinsic Schreier special, with respect to all natural imaginary splittings in $\Lie_{\K}$.\noproof
\end{proposition}

\section*{Acknowledgements}
We would like to thank the referee for comments and suggestions which have led to the present improved version of the text.

\end{document}